\documentclass{birkjour_t2}
\usepackage[utf8]{inputenc}
\usepackage{amsmath,amssymb,paralist,babel}
\usepackage{latexsym,amsfonts,amscd}
\usepackage{graphicx}
\usepackage{color}
\usepackage{dsfont,stix}
\usepackage{comment,doi}
\usepackage[dvipsnames]{xcolor}
\usepackage{tikz}
\usepackage{caption}
\usepackage{subcaption}
\usepackage{comment}
\usepackage{mathtools}
\usetikzlibrary{arrows, backgrounds}

\newtheorem{theo}{Theorem}[section]

\newtheorem{remark}[theo]{Remark}
\newtheorem{assumption}[theo]{Assumption}
\newtheorem{proposition}[theo]{Proposition}

\newtheorem{corol}[theo]{Corollary}

 \numberwithin{equation}{section}

\def\beq{\begin{equation}}
\def\beqn{\begin{eqnarray*}}
\def\eeq{\end{equation}}
\def\eeqn{\end{eqnarray*}}
\def\beqa{\begin{eqnarray}}
\def\eeqa{\end{eqnarray}}

\def\cF{{\mathcal F}}

\def\cA{{\mathcal A}}
\def\calA{{\mathcal A}}

\def\R{\mathbb{R}}

\def\eps{\varepsilon}

\def\u{\tilde{u}}

\def\Q_n{M_n}

\def\Z{\mathbb{Z}}

\def\calV{\mathcal{V}_{d,L}}
\def\calT{\mathcal{T}_{d,L}}

\def\calL{\mathcal{L}}

\def\N{\mathbb{N}}

\DeclareMathOperator{\diam}{diam}

\usetikzlibrary{arrows, backgrounds}

\title[Nonlinear interpolation inequalities and pattern formation in biomembranes]{Nonlinear interpolation inequalities with fractional Sobolev norms and pattern formation in biomembranes}
\author{J. Ginster, A. Pe\v{s}i\'{c}, B. Zwicknagl}
\date{\today}

\begin{document}

\keywords{Scaling law $\mathbf{\cdot}$ interpolation inequalities $\mathbf{\cdot}$ fractional Sobolev norms $\mathbf{\cdot}$ biomembranes $\mathbf{\cdot}$ nonlocal energy $\mathbf{\cdot}$ pattern formation}

\subjclass{46E35, 47J20, 49S05, 74K15}

\begin{abstract}
We consider a one-dimensional version of a variational model for pattern formation in biological membranes. The driving term in the energy is a coupling between the order parameter and the local curvature of the membrane. We derive scaling laws for the minimal energy. As a main tool we present new nonlinear interpolation inequalities that bound fractional Sobolev seminorms in terms of a Cahn-Hillard/Modica-Mortola energy. 
\end{abstract}

\maketitle

\section{Introduction}
We consider a variational model for domain formation in biological membranes. 
We follow the hypothesis that the existence of so-called {\em lipid rafts} is driven by a coupling between the local curvature of the membrane and its local chemical composition. This ansatz forms the basis of a variational model from \cite{Komura:Langmuir:2006} (see also \cite{kawakatsu1993phase2,kawakatsu1993phase,leibler87,Seul:Science:1995,baumgart2003imaging,parthasarathy2006curvature}) which builds on the classical Canham-Helfrich energy \cite{Canham70,Helfrich73}. We focus here on a one-dimensional variant of that, namely 
\begin{equation}\label{eq:functional}
    \mathcal{F}(u,h) := \int_0^1\left( W(u) + \frac{b}{2}|u'|^2 + \frac{\sigma}{2}|h'|^2 + \frac{\kappa}{2}|h''|^2 + \Lambda uh''\right)\,d\calL^1.
\end{equation}
Let us briefly explain the different terms, for a more detailed explanation we refer to the above references or \cite{FonHayLeoZwi}. The function $u\in W^{1,2}((0,1);[-1,1])$ is the order parameter corresponding to the local chemical composition. The values $\pm 1$ represent the pure variants, and we assume for simplicity of notation that $u$ can only take values in $[-1,1]$, corresponding to local mixtures of the components. All our results can easily be generalized to functions with other $L^\infty$-bounds. To enforce some pattern formation, we assume that $u$ has average $0$ (see \eqref{eq:adm} for the precise setting). The first two terms of \eqref{eq:functional} are a Modica-Mortola/Cahn-Hillard-type energy \cite{modica1987gradient} with a  continuous double-well potential $W: [-1,1] \to [0,\infty)$ that has minima at $\{\pm 1\}$ (see Assumption \ref{ass:W} below for the precise conditions), and a term penalizing changes between regions of different composition. The parameter $b>0$ is related to the line tension. The function $h\in W^{2,2}(0,1)$ is the height profile of the membrane, and the parameters $\sigma$ and $\kappa>0$ are related to the surface tension and bending rigidity of the membrane. The term $h''$ stands for the curvature of the membrane, which in a small-slope approximation is given by the Laplacian of the profile function. The last term is the coupling term between the local composition (given by the order parameter $u$) and the local curvature of the membrane (given by the Laplacian of $h$). Note that this term can be negative. The parameter $\Lambda>0$ weights the strength of the coupling. Related coupling terms between the curvature and certain order parameters also occur in the study of surfactants at interfaces between fluids (see e.g. \cite{laradji:2000,brand-dolzmann-pluda:2023}). \\
The functional \eqref{eq:functional} is analytically challenging due to its nonconvex and nonlocal components. Therefore, establishing explicit minimizers analytically in general parameter regimes is a difficult task. There is plenty of evidence that 
the interplay of all terms lead to interesting microstructures in certain parameter regimes (see e.g. \cite{Seul:Science:1995, Komura:Langmuir:2006, ren2004soliton,elliott_hatcher_2021,ELLIOTT20106585,elliott2010surface,elliott2013computation}). Analytical studies of the  functional \eqref{eq:functional}, however, have so far mostly focused on parameter regimes in which pattern formation is not expected (see \cite{FonHayLeoZwi,brazke-et-al,GinHayPesZwi:2024,pesic-pamm-2023}). We note, however, that for a related structurally simpler sharp-interface model, sub- and supercritical parameter regimes for the energy have been identified explicitly in \cite{brazke-diss:2024}.\\
In this article, we follow the approach to study {\em scaling laws} for the infimal energy, that is, we determine the scaling behaviour of the infimal energy in terms of the problem parameters. This in particular involves an ansatz-free lower bound for the minimal energy and the construction of test functions. Roughly speaking, one would expect that for large values of the coupling parameter $\Lambda$, fine-scale patterns are formed, while for small values of $\Lambda$, optimal structures are rather uniform. We make this expectation precise in Theorem \ref{maintheo} in terms of scaling results, where it turns out that it is very subtle to make the statement quantitative in terms of the parameters. A major difficulty in identifying scaling regimes comes from the fact that due to the coupling term the energy can be negative. We focus here only on the scaling behaviour but not on the explicit constants, although they could be tracked through the proofs. Some (non-optimal) choices of constants will be included in A. Pe{\v{s}}i{\'c}' Ph.D. thesis (currently in preparation). We point out that in contrast to the sharp results from \cite{brazke-diss:2024} for a simplified model, we do not cover the full parameter range, and in particular, we leave it to future work to identify critical values of the parameters where the qualitative behaviour of the energy changes.\\
We note that the functional \eqref{eq:functional} bears similarities to well-studied nonlocal functionals like the Ohta-Kawasaki functional \cite{ohta-kawasaki1986}, where the nonlocality comes from a negative fractional Sobolev norm. It has turned out that a useful tool in the proof of scaling laws for such models are interpolation inequalities (see e.g. \cite{Choksi01,ChoksiKohnOtto99,ChoksiKohnOtto2004,choksi-et-al:08,desimone-et-al:2006,cinti-otto:2016,knuepfer-muratov-nolte:2019,brietzke-knuepfer:2023}). We follow a similar approach and present various new non-linear interpolation inequalities in arbitrary space dimensions, bounding fractional Sobolev norms in terms of Modica-Mortola-type functionals (see Section \ref{sec:interpol}). Some of these inequalities are in the spirit of results from \cite{chermisi2010singular,zeppieri41asymptotic}, where higher order nonlinear interpolation inequalities are derived for classical Sobolev seminorms, but the situation for fractional seminorms turns out to be more delicate (see Section \ref{sec:mainresult} for details). We believe that these inequalities are of independent interest and might be useful also when considering high-dimensional versions of \eqref{eq:functional}.
\\
Additionally, we discuss existence and non-existence of minimizers, and show that in some cases where one of the parameters vanishes, the infimum of the energy can be computed exactly (see Section \ref{sec:existence}). The results underline that the qualitative behaviour of the functional comes from a subtle interplay of all terms.  \\ 

The remainder of the paper is organized as follows. After setting some notation, we outline and discuss our main results in Section \ref{sec:mainresult} and state our scaling laws for \eqref{eq:functional} (Theorem \ref{maintheo}) and the new nonlinear interpolation inequalities (Theorem \ref{th:interpol}). We sketch in particular how the latter are used in the proof of the scaling law. After collecting some auxiliary results in Section \ref{sec:prelim}, we prove the interpolation inequalities involving fractional Sobolev seminorms in Section \ref{sec:interpol}. Subsequently, we prove the scaling law \ref{maintheo} in Section \ref{sec:scaling}. Finally, in Section \ref{sec:existence}, we discuss existence of minimizers. \\

\textbf{Notation.} Throughout the note, we use the following notation. For $d\in\N$, we  denote the $d$-dimensional Lebesgue measure by $\calL^d$. We sometimes write $dx$ for $d\calL^d(x)$, and similarly for other integration variables.
For $k\in\N_0$ and $p\in [1,\infty)$, we set 
\begin{eqnarray*}
    W^{k,p}_{\text{per}} &:=& \left\{ h \in W^{k,p}(0,1) : \exists~~ \Tilde{h} \in W^{k,p}_{\textrm{loc}}(\mathbb{R}) ~ 1-\text{periodic such that } h = \Tilde{h}\restriction_{(0,1)} \right\},\\
    W^{k,p}_{\text{per,vol}}&:=&\left\{u\in W^{k,p}_{\text{per}}((0,1);[-1,1]) :\ \int_0^1 u\,d\calL^1=0\right\}.
\end{eqnarray*}
When considering elements in $W^{k,p}((0,1))$, we always refer to their continuous representatives.  Note that for simplicity we restrict ourselves to the case $\int_0^1 u\,d\calL^1=0$. Since the pure phases are represented by $\{u=\pm1\}$, this setting can be seen as the case of equal volume fractions of the two phases. Note that the functional \eqref{eq:functional} does not depend on $h$ itself but only on the derivative of $h$. Hence, we can without loss of generality assume that $\int_0^1h\,d\calL^1=0$. Summarizing, we denote the set of admissible pairs $(u,h)$ for the functional \eqref{eq:functional} by
\begin{eqnarray}
    \label{eq:adm}
\calA:=\left\{(u,h)\in W^{1,2}_{\text{per,vol}} \times W^{2,2}_{\text{per}}:\ \int_0^1 h\,d\calL^1=0\right\}.
\end{eqnarray}
We focus here on periodic boundary conditions for h but we point out that, following the proofs,
similar results for the functional \eqref{eq:functional} can be shown for example for Neumann boundary conditions $h'(0)=h'(1)=0$.
 Finally, we use various light assumptions on the double-well potential $W$, see their discussion in \cite[Section 2]{Choksi01} and Remark \ref{rem:assW} below.
\begin{assumption}
    \label{ass:W}
    \begin{itemize}
        \item[(H1)] There holds $W \in C([-1,1]; [0,\infty))$ with $W^{-1}(0) = \{-1,1\}$.
        \item[(H2)] There is a constant $c_W>0$ such that for all $x\in [-1,1]$, we have $c_W\min\{|x\pm 1|^2\}\leq W(x)$.
      \item[(H3)] There exists a constant $c_W'>0$ and a function $\phi \in C^1([-1,1]) $ such that 
	\begin{eqnarray*}
	\phi'(z) = \sqrt{W(z)},\text{\quad and\quad }
	|z_1 - z_2|^2 \leq c_W' |\phi(z_1) - \phi(z_2)| \quad\text{\ for all\ }z_1,z_2 \in [-1,1].
	\end{eqnarray*}
    \end{itemize}
\end{assumption}

\begin{remark}\label{rem:assW}
Let us note that (H1) and (H2) hold for classical choices of double well-potentials $W$ such as
\[
t \mapsto 1 - t^2,\qquad t \mapsto (1-t^2)^2 \quad \text{ and } \quad t \mapsto- (1-t) \log(1-t)-  (1+t) \log(1+t) + 2\log(2).
\]
Moreover, note that (H1) and (H2) together imply (H3). Indeed, (H2) implies for $-1 \leq z_1 \leq z_2 \leq 1 $ that
\begin{align*}
|\phi(z_1)- \phi(z_2)| &\geq \sqrt{c_W} \int_{z_1}^{z_2} \min\{ 1\pm t \} \, dt = \sqrt{c_W} \int_{z_1}^{z_2} (1 - |t|) \, dt \\ &= \sqrt{c_W} \left( z_2 - z_1- \frac12 \left(z_2^2\, \text{sign}(z_2) - z_1^2 \,\text{sign}(z_1)\right) \right).
\end{align*}
Assume first that $\text{sign}(z_2) = \text{sign}(z_1)$. Without loss of generality, we consider $\text{sign}(z_2) = 1$. Then
\[
|\phi(z_1)- \phi(z_2)| \geq \sqrt{c_W} (z_2 - z_1) \left( 1 - \frac12  (z_2+ z_1)  \right) =  \frac{\sqrt{c_W}}2 (z_2 - z_1) \left( |1-z_2| + |1-z_1| \right)\geq \frac{\sqrt{c_W}}2 |z_2 - z_1|^2.
\]
On the other hand, if $\text{sign}(z_2) = 1 = - \text{sign}(z_1)$ then 
\[
|\phi(z_1)- \phi(z_2)| \geq \sqrt{c_W} \left(z_2 - z_1- \frac12 z_2^2 - \frac12 z_1^2 \right) \geq \frac{\sqrt{c_W}}2 (z_2 - z_1) \geq \frac{\sqrt{c_W}}2 |z_2 - z_1|^2.
\]
\end{remark}

\subsection{Overview and main results}\label{sec:mainresult}
In this section we explain our main results. 
Our first main result is the following scaling law for \eqref{eq:functional}.
\begin{theo}\label{maintheo}
Suppose that $W$ satisfies Assumption \ref{ass:W}. Then there exist constants $C\geq c > 0$ such that for all $\Lambda, \kappa, \sigma, b > 0$ the following holds.
If $\Lambda^2 \geq C \max\left\{b\sigma, b \kappa, (b\sigma\kappa)^{1/2}, b^{1/2}\kappa\right\}$ then
\begin{equation}\label{eq:sc1}
    - \frac{\Lambda^2}{2\kappa} \leq \inf_{\calA}\mathcal{F} \leq -c \frac{\Lambda^2}{\kappa}.
\end{equation}
If $\Lambda^2 \leq c \max\left\{b\sigma, b \kappa, (b\sigma\kappa)^{1/2}, b^{1/2}\kappa\right\}$ then 
\begin{equation}\label{eq:sc2}
    c \min\{b^{1/2},1\} \leq \inf_{\calA}\mathcal{F} \leq C \min\{b^{1/2},1\}.
\end{equation}
\end{theo}

\begin{figure}[htbp]
\centering
\begin{subfigure}[b]{0.45\textwidth}
\includegraphics[width=.9\textwidth]{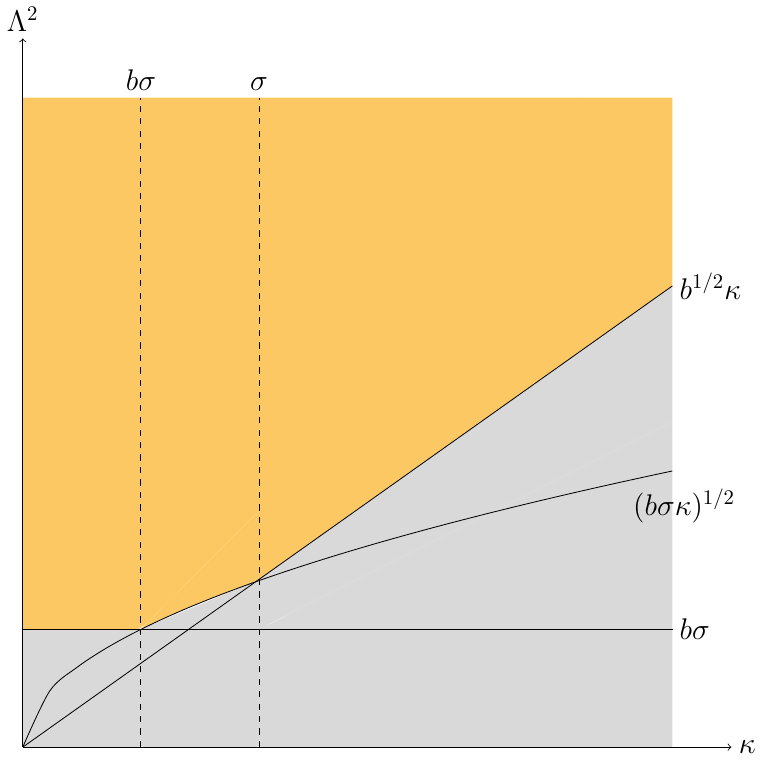}
\caption{Case $b<1$}\label{fig:b<1}
\end{subfigure}
\begin{subfigure}[b]{0.45\textwidth}
    \centering
\includegraphics[width=.9\textwidth]{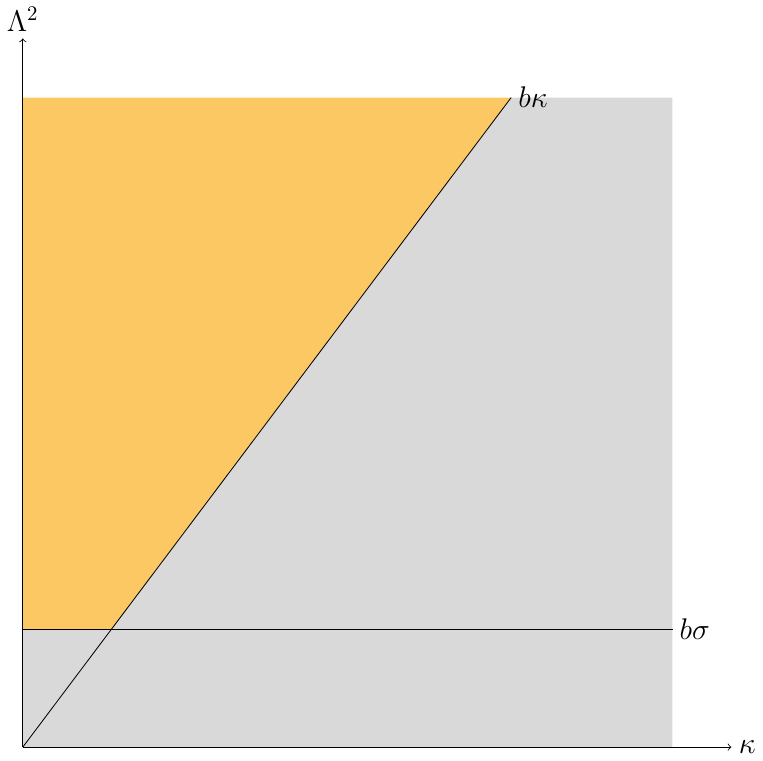}
\caption{Case $b>1$}\label{fig:b>1}
\end{subfigure}
 \caption{Sketch of the scaling regimes; regions in orange indicate parameter regimes with scaling $-\Lambda^2/\kappa$, while the gray regions correspond to regimes with scaling $\min\{b^{1/2}, 1\}$.}
    \label{fig:main}
\end{figure}
We point out that there are only three scalings of the infimal energy. While the regimes themselves depend on all parameters, the scalings of the energy do not depend on $\sigma$.
Let us briefly outline the main ideas of the proof of the scaling laws in Theorem \ref{maintheo}. 
For the upper bounds, we use three different types of test functions. To obtain the scalings $\min\{1, b^{1/2}\}$, it suffices to consider flat height profiles $h\equiv 0$ and rather uniform structures for the order parameter $u$, namely $u\equiv 0$ for $b\geq 1$ and a function $u$ that has two transition layers of length $\sim b^{1/2}$ between regions where $u=1$ and ones where $u=-1$,  see Figure \ref{fig: construction1}. Eventually, to obtain the scaling of the form $-\frac{\Lambda^2 }{\kappa}$ we make the ansatz that $h$ has a piecewise constant curvature $h''$ whose sign oscillates and a function $u$ which - up to a transition layer - is given by the negative of the sign of $h''$, see Figure \ref{fig: construction2}. In this way, the last term of the energy is essentially given by $\Lambda \int_{0}^1 h'' u \,d\calL^1 \sim -\Lambda \int_0^1|h''|\, d\calL^1$. Optimizing this only versus the term $\int_0^1 \kappa |h''|^2\, d\calL^1$ suggests to choose $|h''| = \frac{\Lambda}{2 \kappa}$ to obtain $\int_0^1 \kappa |h''|^2 + \Lambda h'' u \, d\calL^1 \sim -\frac{\Lambda^2}{2 \kappa}$. Incorporating the other terms of the energy makes this optimization more complex and includes a joint optimization in the number of oscillations and the transition length for $u$. For example, increasing the number of oscillations decreases the influence of the term $\int_0^1 \sigma |h'|^2 \, d \calL^1$ but will increase the energy coming from the terms in $u$ as every transition of $u$ from $-1$ to $+1$ (and vice versa) induces a certain amount of energy (depending on $b$). This is made precise in Section \ref{sec:ub}.

\begin{figure}
\centering
\begin{subfigure}[b]{0.4\textwidth}
\centering
\includegraphics[width=0.9\textwidth]{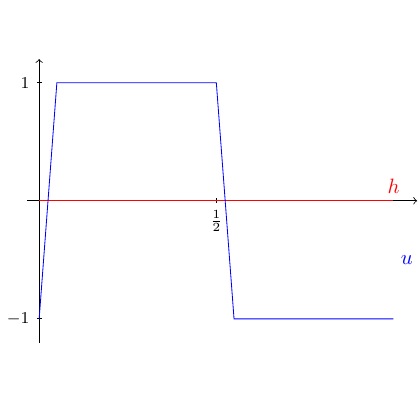}
\caption{Sketch of the configuration with a flat $h$ (red) and $u$ (blue) with one transition that satisfies $\mathcal{F}(u,h) \sim b^{1/2}$ for $0 < b \leq 1$. }
\label{fig: construction1}
\end{subfigure}
\qquad
\begin{subfigure}[b]{0.4\textwidth}
\centering
\includegraphics[width=0.9\textwidth]{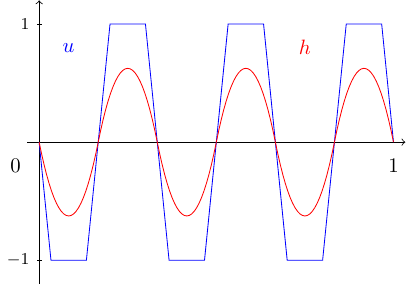}
\caption{Sketch of the regime where $\Lambda^2$ is suffciciently large. The function $h$ (red) has piecewise constant curvature with changing sign. Up to a transition layer the function $u$ (blue) is the negative of the sign of $h''$. }\label{fig: construction2}
\end{subfigure}
\end{figure}
The proof of the lower bound turns out to be more involved. The difficulty arises from the nonlocal coupling term, which generally does not have a sign. The lower bound $-\frac{\Lambda^2}{2\kappa}$ follows from Young's inequality
and the assumption that $|u|\leq 1$ (see Proposition \ref{prop:lbyoung}). On the other hand, the lower bound $\min\{b^{1/2},1\}$ is the expected scaling from the Modica Mortola energy,
\[\int_0^1\left(W(u)+\frac{b}{2}|u'|^2\right)\,d\calL^1\geq c_{MM}\min\{1,\,b^{1/2}\},\]
 see Proposition \ref{prop:modica-mortola}.
 The main task therefore is to show that for small $\Lambda>0$ (see the assumptions for \eqref{eq:sc2}), the coupling term is controlled by the other terms in the energy.
  Let us briefly remark here that these assumptions on the smallness of $\Lambda$ are significantly weaker than simply assuming $\frac{\Lambda^2}{2\kappa} \lesssim \min\{1,b^{1/2}\}$ which would ensure the lower bound in \eqref{eq:sc2} by the arguments above.
In particular, the statements of Theorem \ref{maintheo} do not imply that there are constants $c_1,\dots, c_4$ such that for all parameters $b,\sigma,\kappa, \Lambda>0$
\[c_1\min\{1,b^{1/2}\}-c_2\frac{\Lambda^2}{2\kappa}\leq \inf_\calA \cF\leq c_3\min\{1,b^{1/2}\}-c_4\frac{\Lambda^2}{2\kappa}. \]

In order to understand the argument for the lower bound, let us first optimize in $h$ to obtain the following nonlocal functional for $u$ (see e.g. ~\cite[Appendix]{FonHayLeoZwi} and the proof of Corollary \ref{cor:lbinterpol})
\begin{equation}\label{eq: nonlocal u}
\mathcal{F}(u,h) \geq \int_0^1\left( W(u) + \frac{b}2 |u'|^2 \right)\, d\calL^1 - \frac{\Lambda^2}{2 \kappa} \sum_{k \in \Z }  \frac{\kappa k^2}{\sigma + \kappa k^2} |\hat{u}_k|^2.
\end{equation}
Here, $\hat{u}_k$ denotes the $k$-th Fourier coefficient of $u$.
Hence, proving a lower bound corresponds to bounding the third term by the first two terms.
A first naive bound on the last term is given by 
\[
\frac{\Lambda^2}{2 \kappa} \sum_{k \in \Z} \frac{\kappa k^2}{\sigma + \kappa k^2} |\hat{u}_k|^2 \leq \frac{\Lambda^2}{2 \sigma} \sum_{k \in \Z }  k^2 |\hat{u}_k|^2 = \frac{\Lambda^2}{2 \sigma} \int_0^1 |u'|^2 \, d\calL^1.
\]
This yields the desired estimate if $\Lambda^2 \lesssim b \sigma$. Similarly, one can bound 
\[
\frac{\Lambda^2}{2 \kappa} \sum_{k \in \Z }  \frac{\kappa k^2}{\sigma + \kappa k^2} |\hat{u}_k|^2 \leq \frac{\Lambda^2}{2 \kappa} \sum_{k \in \Z }  |\hat{u}_k|^2 = \frac{\Lambda^2}{2 \kappa} \int_0^1 |u|^2 \, d\calL^1 \leq \frac{\Lambda^2}{2\kappa},
\]
which is essentially the same bound as in Proposition \ref{prop:lbyoung} and implies the needed control if, for example, $\Lambda^2 \gtrsim \min\{b^{1/2},1\} \kappa$.
However, in other regimes these naive bounds are not sufficient. 
At first glance, it might appear to be useful to underline the nonlocal nature of the third term in \eqref{eq: nonlocal u} by drawing a connection to a fractional Sobolev norm by estimating
\[
\frac{\Lambda^2}{2 \kappa} \sum_{k \in \Z }  \frac{\kappa k^2}{\sigma + \kappa k^2} |\hat{u}_k|^2 \lesssim \frac{\Lambda^2}{2 \kappa} \sum_{k \in \Z }  \frac{\kappa^{1/2}}{\sigma^{1/2}} |k| |\hat{u}_k|^2 = \frac{\Lambda^2}{\kappa^{1/2}\sigma^{1/2}} [u]_{H^{1/2}}^2.
\]
Then the needed control of the third term in \eqref{eq: nonlocal u} could be established through a nonlinear interpolation inequality of the form
\begin{equation}\label{eq: nonl interpol}
[u]_{H^{1/2}}^2 \leq C \int_0^1\left( \frac1{\delta} W(u) + \delta |u'|^2\right) \, d\calL^1.
\end{equation}
That such an inequality might hold could be suspected as the linear interpolation inequality, i.e., replacing $W(u)$ by $|u|^2$, is immediate (see \eqref{eq:interpollinear}). Moreover, if one replaces the fractional Sobolev space $H^{1/2}$ by $H^1$ a similar inequality was shown in \cite[Lemma 3.1]{zeppieri41asymptotic} and \cite[Theorem 1.2]{chermisi2010singular}. Precisely, for any connected open domain $\Omega\subseteq \R^d$ there are $\delta_0>0$ and $q>0$ such that for all $0<\delta\leq \delta_0$ and all $u\in W^{2,2}_{\text{loc}}(\Omega)$ there holds
 \begin{eqnarray}
 \label{eq:interpolh1}q\|Du\|_{L^2(\Omega)}^2\leq \int_\Omega\left(\frac{1}{\delta} W(u)+\delta\left|D^2 u\right|^2\right)\,d\calL^d. 
 \end{eqnarray}
However, we prove that \eqref{eq: nonl interpol} holds only with an extra factor $|\ln(\delta)|$ on the right hand side. Precisely, we prove (see Proposition \ref{prop:interpolation}) a nonlinear interpolation inequality of the form
\begin{equation}\label{eq: logarithmic}
c_L |u|_{H^{1/2}(\Pi^d)}^2 \leq |\ln \delta| \int_{\Pi^d} \left(\frac{1}{\delta} W(u) + \delta |\nabla u|^2\right)\, d\calL^d,
\end{equation}
where $\Pi^d$ denotes the $d$-dimensional torus. We also provide an example which shows that this estimate is sharp in $\delta$ (see Remark \ref{rem:h12}). Moreover, we establish the following similar nonlinear interpolation inequalities for fractional Sobolev spaces also in the non-periodic setting which are of independent interest (see Section \ref{sec:interpol}). 
\begin{theo}\label{th:interpol}
    Let $W$ satisfy (H1) and (H2), $d \in \N$, and let $\Omega \subseteq \R^d$ be open, convex and bounded.
    \begin{enumerate}
    \item If $s \in (0,1/2)$ then there exists a constant $c > 0$ such that for all $u \in W^{1,2}(\Omega)$ and all $\delta\in(0,1/2)$ there holds
   \begin{equation*}
       c \int_{\Omega} \int_{\Omega} \frac{|u(x) - u(y)|^2}{|x-y|^{d+2s}} \, dx\,dy \leq \int_{\Omega} \left(\frac{1}{\delta} W(u) + \delta |\nabla u|^2\right)\, d\calL^d.
   \end{equation*}  
   \item If $s =1/2$ then there exists a constant $c > 0$ such that for all $u \in W^{1,2}(\Omega)$ and all $\delta\in(0,1/2)$ there holds
   \begin{equation*}
       c \int_{\Omega} \int_{\Omega} \frac{|u(x) - u(y)|^2}{|x-y|^{d+1}} \, dx\,dy \leq |\ln\delta|\int_{\Omega} \left(\frac{1}{\delta} W(u) + \delta |\nabla u|^2\right)\, d\calL^d.
   \end{equation*}  
   \item If $s\in (1/2,1)$ then there exists a constant $c>0$ such that for all $u \in W^{1,2}(\Omega)$ and all $\delta\in(0,1/2)$ there holds   \begin{equation*}
        c \int_{\Omega} \int_{\Omega} \frac{|u(x) - u(y)|^2}{|x-y|^{d+2s}} \, dx\,dy \leq  \delta^{1-2s}\int_{\Omega} \left(\frac{1}{\delta} W(u) + \delta |\nabla u|^2\right)\, d\calL^d.
   \end{equation*} 
   \end{enumerate}
   \end{theo}

We present the statements here only for functions in $W^{1,2}(\Omega)$ but following the lines of the proofs one can also obtain analogous bounds on fractional $W^{s,p}$-seminorms for functions $u\in W^{1,p}(\Omega)$ for more general $p\in (1,\infty)$ for suitably adapted double-well potentials. We note that the interpolation results hold in arbitrary space dimensions. We hope that they will be useful for the study of higher-dimensional variants of \eqref{eq:functional}.
Returning to the lower bound in Theorem \ref{maintheo}, let us remark that the logarithmic factor in the nonlinear fractional interpolation inequality \eqref{eq: logarithmic} implies that using this inequality can only provide the desired control under stronger conditions on the smallness of $\Lambda$ than assumed in \eqref{eq:sc2}.
In order to overcome this problem we prove for the specific expression of our problem the following sharper inequality without a logarithmic factor. 
\begin{proposition}\label{prop:basicinterpolwithoutdelta}
Let $d>1$, and suppose that $W$ satisfies Assumption \ref{ass:W}. Then there exists a constant $c>0$ such that for all $L\in\N$, all $M,\delta>0$, and all $u\in W^{1,2}_{\text{per}}(\Pi^d)$ there holds
\begin{eqnarray*}
\sum_{k\in\Z^d}\min\{1,\frac{|k|^2}{M^2}\}|\hat{u}_k|^2\leq c\left(\frac{1}{L}+\frac{L}{M^2}\right)\int_{\Pi^d}\left(\frac{1}{\delta}W(u)+\delta|\nabla u|^2\right)\,d\calL^d.
\end{eqnarray*}
\end{proposition}
We note that a similar term to the one on the left-hand side is also estimated in \cite[Lemma 1]{desimone-et-al:2006} but by different quantities on the right-hand side, and in particular not by a Modica-Mortola-type energy. Moreover, the term $\sum_{k\in\Z^d}|k|^2|\hat{u}_k|^2$ corresponds to the Laplacian of $u$. We therefore believe that this estimate can also be useful to study higher-dimensional variants of \eqref{eq:functional} involving an approximated bending term $\frac{\kappa}{2}\|\Delta h\|_{L^2}^2$.

\section{Preliminaries}\label{sec:prelim}
Throughout the proof of the scaling law, the following two lower bounds on parts of the energy \eqref{eq:functional} will be frequently used. The first one is a well-known lower bound for the classical Modica-Mortola energy, which we recall for the readers' convenience.
\begin{proposition}\label{prop:modica-mortola}
 Let $W$ satisfy $(H1)$. There is a constant $c_{MM}>0$ such that for all $\eps>0$ and all $u\in W^{1,2}_{\text{per,vol}}$ 
 \[ \int_0^1 \left(W(u)+\eps^2|u'|^2\right)\,d\calL^1\geq c_{MM}\min\{1,\eps\}. \]
 We can choose $c_{MM}:=\max\left\{\min_{[-1/2,1/2]}W,\ \left( \min_{[-1/2,1/2]} W \right)^{1/2}\right\}$.
\end{proposition}
\begin{proof}
  Let $\eps>0$ and $u\in W^{1,2}_{\text{per,vol}}$. If $\max |u|\leq \frac{1}{2}$ then 
  \[\int_0^1 \left(W(u)+\eps^2|u'|^2\right)\,d\calL^1\geq\min_{t\in [-1/2,1/2]}W(t). \]
  Next, assume there exists $t_0\in (0,1)$ with $u(t_0)>1/2$. Since $\int_0^1 u \, dx =0$ and u is continuous, there is $t_1\in (0,1)$ such that $u(t_1)<0$.  Without loss of generality, we may assume that that $t_0<t_1$ (the case $t_1 < t_0$ can be treated similarly). Then 
\begin{eqnarray*}
\int_0^1 \left( W(u) + \varepsilon^2|u'|^2\right) \,d\calL^1 &\geq& \int_{t_0}^{t_1}\left(W(u) + \varepsilon^2|u'|^2 \right)\,d\calL^1\geq 2\eps \int_{t_0}^{t_1}
        \sqrt{W(u)} u'\, d\calL^1\\
        &=& 2\eps\int_{u(t_0)}^{u(t_1)} \sqrt{W(t)}\,dt
        \geq 2\eps \int_0^{1/2} \sqrt{W(t)}\,dt\\
        & \geq& \eps \left( \min_{[0,1/2]} W \right)^{1/2}.
    \end{eqnarray*}
    Eventually, we argue analogously if there exists $t_0\in (0,1)$ with $u(t_0)<-1/2$.
\end{proof}
On the other hand, we have the following lower bound on the nonlocal coupling term. 
We note that the constant $-1/2$ in the lower bound in the following Proposition is optimal, see Remark \ref{rem:youngopt} below.
\begin{proposition}\label{prop:lbyoung}
    For all parameters  $\Lambda, \kappa,\sigma, b>0$ and all Borel measurable functions $W: [-1,1]\to [0,\infty)$ 
    it holds
    \begin{equation*}
        -\frac{\Lambda^2}{2\kappa} \leq \inf_{\calA} \mathcal{F}.
    \end{equation*}
\end{proposition}

\begin{proof}
    Let $(u,h)\in \calA$ be arbitrary. Then by Young's inequality it holds
    \begin{equation*}
    \begin{split}
        \left| \int_0^1 \Lambda u h'' \,d\calL^1 \right| &\leq \int_0^1 \Lambda |u| |h''| \,d\calL^1
        \leq \int_0^1 \left(\frac{\Lambda^2 u^2}{2\kappa} + \frac{\kappa}{2}|h''|^2\right)\, d\calL^1
        \leq \frac{\Lambda^2}{\kappa} + \frac{\kappa}{2} \int_0^1 |h''|^2\, d\calL^1,
    \end{split}
    \end{equation*}
    where the last inequality follows from the boundedness assumption $|u|\leq 1$. Therefore,
    \begin{equation*}
        \begin{split}
        \mathcal{F}(u,h) &\geq \int_0^1 \left(W(u) + \frac{b}{2}|u'|^2 + \frac{\sigma}{2}|h'|^2 + \frac{\kappa}{2}|h''|^2\right) \,d\calL^1 - \left| \int_0^1 \Lambda uh'' d\calL^1 \right|\\
        &\geq \int_0^1 \left(W(u) + \frac{b}{2}|u'|^2 + \frac{\sigma}{2}|h'|^2\right) \,d\calL^1 - \frac{\Lambda^2}{2\kappa}\\
        &\geq - \frac{\Lambda^2}{2\kappa}.
        \end{split}
    \end{equation*}
\end{proof}

\section{Interpolation inequalities}\label{sec:interpol}
In this section, we present some new interpolation-type arguments involving fractional-order Sobolev semi-norms. The results here hold in arbitrary space dimension, and are later applied in $d=1$ in the proof of the Theorem \ref{maintheo}. We denote by $\Pi^d$ the $d$-dimensional torus, and by $C$ generic positive constants that may change from line to line. Sometimes, we specify dependences of the constant on certain parameters $\lambda$ by indices $C_\lambda$ or $C(\lambda)$.
\subsection{Linear interpolation with fractional Sobolev seminorms}
We first recall some basics on fractional-order Sobolev seminorms for periodic functions. 
A periodic function $u\in L^2(\Pi^d)$ can be represented as Fourier sum via
\begin{eqnarray}
\label{eq:Fourierseries}
    u(x)=\sum_{k\in\Z^d}\hat{u}_ke^{2\pi i k\cdot x}\text{\qquad with\qquad}\hat{u}_k:=\int_{\Pi^d} u(t) e^{-2\pi i k\cdot t}\, dt,
\end{eqnarray}
where we denote by $a\cdot b$ the Euclidean scalar product of vectors $a,b\in\mathbb{C}^d$.
Consider 
$u\in W^{1,2}_{\text{per}}(\Pi^d)$ and $s\in (0,1]$. Then the (fractional) Sobolev semi-norm is defined as
\[|u|_{H^s(\Pi^d)}^2=\sum_{k\in\Z^d}|k|^{2s}|\hat{u}_k|^2. \]
By H\"older's inequality, we have the (linear) interpolation estimate
\begin{eqnarray}\label{eq:interpollinear}
|u|_{H^s(\Pi^d)}^{2}&=&\left(\sum_{k\in\Z^d}|\hat{u}_k|^{2(1-s)}\cdot |k|^{2s}|\hat{u}_k|^{2s}\right)
\leq \left(\sum_{k\in\Z^d}|\hat{u}_k|^2\right)^{1-s}\left(\sum_{k\in\Z^d}|k|^2|\hat{u}_k|^2\right)^{s}\nonumber\\&=&\|u\|_{L^2(\Pi^d)}^{2(1-s)} |u|_{H^1(\Pi^d)}^{2s} = \delta^{-2(1-s) s} \|u\|_{L^2(\Pi^d)}^{2(1-s)} \delta^{2 (1-s) s}|u|_{H^1(\Pi^d)}^{2s}\nonumber\\&\leq& \delta^{-2s} \| u \|_{L^2(\Pi^d)}^2 + \delta^{2(1-s)} |u|_{H^1(\Pi^d)}^2  . \label{eq: linear interpolation}
\end{eqnarray}
By \cite[Proposition 1.3]{BenyiOh}, there is an (equivalent) integral representation (cf. also, e.g.,  \cite{hoermander03,HitchFracSob} for similar characterizations on full space)
 \begin{equation}\label{eq:equivsemi}
 c_{\text{Fl}}  \|u\|_{H^{s}(\Pi^d)}^2\leq  \int_{\Pi^d} \int_{[-1/2,1/2)^d}\frac{|u(x+y)-u(y)|^2}{|x|^{d+2s}}\,dx\,dy    \leq c_{Fu} \|u\|_{H^{s}(\Pi^d)}^2.
    \end{equation}
   Following the proof there, the constants can be chosen as 
    \begin{eqnarray*}
    c_{\text{Fu}}&:=&\int_{\R^d}\frac{\sin^2(\pi x_1)}{4|x|^{d+2s}}\,dx\text{\qquad and\qquad}
    c_{\text{Fl}}:=2^{2s-3-d}d^{-5}(1-s)^{-1} \frac{\pi^{(d-1)/2}}{\Gamma\left(\frac{(d-1)}{2} + 1\right)}.
    \end{eqnarray*}
\subsection{Nonlinear interpolation with fractional Sobolev seminorms}\label{sec:nonlininterpol1}
As discussed in Section \ref{sec:mainresult}, we will need counterparts of interpolation inequalities of the form \eqref{eq:interpollinear} where the $L^2$-norm is replaced by a term involving $W$. This has been used in the study of ''local'' approximations of \eqref{eq:functional} (see \cite[Lemma 3.1]{zeppieri41asymptotic} and \cite[Theorem 1.2]{chermisi2010singular}), see \eqref{eq:interpolh1}) and its discussion in Section \ref{sec:mainresult}.
However, for fractional Sobolev seminorms, the situation turns out to be slightly more subtle, as the following example shows. 
\begin{remark}\label{rem:h12}
 Suppose that $W$ satisfies (H1). We claim that there is no constant $c>0$ such that for all $u\in W^{1,2}_{\text{per}}$ and all $\delta>0$ there holds 
 \begin{eqnarray}\label{eq:interpolh12}
 c|u|_{H^{1/2}(0,1)}^2\leq \int_{0}^1 \left(\frac{1}{\delta} W(u)+\delta |u'|^2 \right)\,d\calL^1.\end{eqnarray}
 Indeed, consider for $0<\delta<1/8$ the function $u_\delta\in W^{1,2}_{\text{per}}((0,1);[-1,1])$ given by
 \[u_\delta(x):=\begin{cases}
     -\frac{2}{\delta}x+1&\text{\qquad if \ }x\in (0,\delta),\\
     -1&\text{\qquad if \ }x\in (\delta, \frac{1}{2}-\frac{\delta}{2}),\\
     \frac{2}{\delta}x-\frac{1}{\delta}&\text{\qquad if\ }x\in (\frac{1}{2}-\frac{\delta}{2},\frac{1}{2}+\frac{\delta}{2}),\text{\ and}\\
     1&\text{\qquad if\ }x\in (\frac{1}{2}+\frac{\delta}{2},1).
 \end{cases}\]
 Then the right-hand side of \eqref{eq:interpolh12} is uniformly bounded above since
 \[\int_0^1 \left(\frac{1}{\delta} W(u) + \delta |u'|^2\right)\, d\calL^1\leq \frac{2\delta}{\delta} \max W+\delta\cdot 2\delta\left(\frac{2}{\delta}\right)^2=2\max W+8.\]
 However, using \eqref{eq:equivsemi} the left-hand side is estimated below via (see, e.g., \cite{conti-zwicknagl:16} for a similar computation)
 \begin{eqnarray*}
 &&\int_0^1\int_{-1/2}^{1/2}\frac{|u(x+y)-u(y)|^2}{|x|^{2}}\,dx\,dy  \geq\int_{1/4}^{(1-\delta)/2}\int_{(1+\delta)/2-y}^{1/2}\frac{|u(x+y)-u(y)|^2}{|x|^{2}}\,dx\,dy\\
 &=&\int_{1/4}^{(1-\delta)/2}\int_{(1+\delta)/2-y}^{1/2}\frac{4}{|x|^{2}}\,dx\,dy=4\int_{1/4}^{(1-\delta)/2}\left(-2+\frac{1}{\frac{1+\delta}{2}-y}\right)\,dy\\
 &=&-8\left(\frac{1}{4}-\frac{\delta}{2}\right)-4\ln\left(\frac{\delta}{\frac{1}{4}+\frac{\delta}{2}}\right)\longrightarrow \infty\text{\qquad as\ }\delta\to 0.
 \end{eqnarray*}
\end{remark}
While Remark \ref{rem:h12} shows that there is no interpolation inequality of the form \eqref{eq:interpolh12}, we show below that a corresponding estimate holds for all fractional $H^s$-Sobolev norms with $s\in(0,1/2)$ (see Proposition \ref{prop:interpolation}), and that the logarithmic correction observed in Remark \ref{rem:h12} is the worst case in the interpolation inequality for the $H^{1/2}$-seminorm.

\begin{proposition}
 \label{prop:interpolation}
   Suppose that $W$ satisfies (H1) and (H2), and let $d \in \N$. 
   \begin{enumerate}
   \item If $s\in (0,1/2)$ then there 
   exists a constant $c_{L,s} > 0$ such that for all $u \in W^{1,2}_{\text{per}}(\Pi^d)$ and all $\delta\in(0,1/2)$ there holds
   \begin{equation*}
       c_{L,s} |u|_{H^s(\Pi^d)}^2 \leq \int_{\Pi^d} \left(\frac{1}{\delta} W(u) + \delta |\nabla u|^2\right)\, d\calL^d.
   \end{equation*}    
   \item For $s=1/2$, there exists a constant $c_L > 0$ such that for all $u \in W^{1,2}_{\text{per}}(\Pi^d)$ and all $\delta\in(0,1/2)$ there holds   \begin{equation}\label{eq:nonlininterpol12}
       c_L |u|_{H^s(\Pi^d)}^2 \leq |\ln \delta| \int_{\Pi^d} \left(\frac{1}{\delta} W(u) + \delta |\nabla u|^2\right)\, d\calL^d.
   \end{equation} 
   \item If $s\in (1/2,1)$ then there exists a constant $c_{L,s}>0$ such that for all $u \in W^{1,2}_{\text{per}}(\Pi^d)$ and all $\delta\in(0,1/2)$ there holds   \begin{equation}\label{eq:nonlininterpol_gtr12}
       c_{L,s} |u|_{H^s(\Pi^d)}^2 \leq  \delta^{1-2s}\int_{\Pi^d} \left(\frac{1}{\delta} W(u) + \delta |\nabla u|^2\right)\, d\calL^d.
   \end{equation} 
   \end{enumerate}
\end{proposition} 

\begin{remark}
We note that the scaling in $\delta$ for $s> \frac12$, c.f.~\eqref{eq:nonlininterpol_gtr12}, is the same as in the linear interpolation inequality \eqref{eq: linear interpolation}. We point out, however, that the same cannot be true for $0< s < \frac12$. Indeed, for the function $u_{\delta}$ in Remark \ref{rem:h12}  and $0 < \delta < 1/5$ it holds
\begin{align*}
| u_{\delta} |_{H^{s}}^2 &\geq c_{Fl} \int_0^1 \int_{-1/2}^{1/2} \frac{ |u_{\delta}(x+y) - u_{\delta}(y)|^2}{x^{1+2s}} \, dx dy \geq \int_{1/4}^{(1-\delta)/2} \int_{(1+\delta)/2-y}^{1/2} \frac{4}{|x|^{1+2s}} \, dx dy \\ &\geq \int_{1/4}^{2/5} \int_{1/4+\delta/2}^{1/2} \frac{4}{|x|^{1+2s}} \, dx dy \geq   4 (2/5-1/4) \cdot (1/2 - 1/4 - 1/10) \geq \frac{9}{100}, 
\end{align*}
whereas it holds as in Remark \ref{rem:h12}
\[
\delta^{1-2s} \int_0^1 \left(\frac1{\delta} W(u_{\delta}) + \delta |u_{\delta}'|^2\right) \, d\calL^1 \leq \delta^{1-2s} \left( 2 \max W + 8\right) \longrightarrow 0 \text{ as } \delta \to 0.
\]

\end{remark}

\begin{proof}[Proof of Proposition \ref{prop:interpolation}]
Large parts of the proof agree for all three cases. We therefore consider them together and comment on the differences. 
    Let $u \in W^{1,2}_{per}(\Pi^d) \cap C^1(\Pi^d)$, $s\in (0,1/2]$, and $\delta\in (0,1/2)$.
    By \eqref{eq:equivsemi}, it suffices to estimate 
    \begin{equation}\label{eq:splitinterpol}
        \begin{split}
          & 
            c_{\text{Fl}} |u|_{H^s(\Pi^d)}^2 \leq  \int_{\Pi^d} \int_{[-1/2,1/2)^d} \frac{|u(x+y) - u(y)|^2}{|x|^{d+2s}}\,dx\, dy\\
            &= \int_{\Omega_1}\frac{|u(x+y) - u(y)|^2}{|x|^{d+2s}}\,dx,dy
            +   \int_{\Omega_2}\frac{|u(x+y) - u(y)|^2}{|x|^{d+2s}}\,dx dy+\int_{\Omega_3}\frac{|u(x+y) - u(y)|^2}{|x|^{d+2s}}\,dx dy
        \end{split}
    \end{equation}
    with 
    \begin{eqnarray*}
        \Omega_1&:=&(-\delta,\delta)^d\times\Pi^d\nonumber,\\
        \Omega_2&:=&\left\{(x,y)\in [-\frac{1}{2},\frac{1}{2})^d\times \Pi^d:\ \ |x|\geq \delta\text{\ and\ } |u(x+y)-u(y)|\leq 1/4 \text{ or } |u(x+y)-u(y)| \geq 4 \right\}, \text{\ and}\nonumber\\
        \Omega_3&:=&\left\{(x,y)\in [-\frac{1}{2},\frac{1}{2})^d\times \Pi^d:\ |x|\geq \delta\text{\ and\ }\ 4 > |u(x+y)-u(y)|> 1/4\right\}.
    \end{eqnarray*}
    We consider the three terms in \eqref{eq:splitinterpol} separately. The contribution from $\Omega_1$ is estimated by 
    \begin{equation}\label{eq:interpolI1}
        \begin{split}
          \int_{\Omega_1}\frac{|u(x+y) - u(y)|^2}{|x|^{d+2s}}\,dx dy
             &= \int_{\Pi^d} \int_{(-\delta,\delta)^d}\frac{1}{|x|^{d+2s}}\left(\int_0^1 \nabla u(y+tx) \cdot x \,  dt\right)^2\,dx\,dy\\
            &\leq \int_{\Pi^d} \int_{(-\delta,\delta)^d} |x|^{2-d-2s} \int_0^1 |\nabla u(y+tx)|^2\, dt\, dx\, dy\\
            &\leq  \int_{(-\delta,\delta)^d} |x|^{2-d-2s} \int_0^1 \int_{ tx + \Pi^d} |\nabla u (w)|^2\, dw\, dt\, dx \\
            &\leq C(d)\delta^{2(1-s)} \int_{\Pi^d} |\nabla u|^2\, d\calL^d \leq C(d) \max\{1,\delta^{1-2s}\}\delta \int_{\Pi^d} |\nabla u|^2\, d\calL^d.
        \end{split}
    \end{equation}
   Consider now the contribution from $\Omega_2$. We show first that for all $(x,y) \in \Omega_2$,
   \begin{equation}\label{eq: est diff u}
   |u(x+y)-u(y)|^2\leq 18 \left(\min\{|u(x+y)\pm 1|^2\}+\min\{|u(y)\pm 1|^2\}\right).
   \end{equation}
   Assume first that $|u(x+y) - u(y)| \leq 1/4$. Let $e \in \{\pm 1\}$ be such that $|u(x+y) - e| = \min\{ |u(x+y) \pm 1|\}$. 
   It follows that $|u(y) - e| \leq 2 \min\{ |u(y) \pm 1|\}$. 
   Consequently, 
   \[
   |u(x+y) - u(y)|^2 \leq 2 \left( |u(x+y) - e|^2 + |u(y) - e|^2 \right) \leq 8 \left( \min\{ |u(x+y) \pm 1|^2 \} + \min\{|u(y) \pm 1|^2 \} \right).
   \]
   Now assume that $|u(x+y) - u(y)| \geq 4$. Then $\min\{|u(y) \pm 1|\} \geq 1$ or $\min\{|u(x+y) \pm 1|\} \geq 1$. Without loss of generality, we assume that the latter is true. 
   Let $E \in \{\pm 1\}$ be such that $|u(x) - E | = \min\{|u(x) \pm 1| \}$. Then $|u(x+y) - E| \leq \min\{ |u(x+y) \pm 1| \} + 2 \leq 3 \min\{ |u(x+y) \pm 1| \}$. 
   Hence, 
   \[
   |u(x+y) - u(y)|^2 \leq 2 \left( |u(x+y) - E|^2 + |u(y) - E|^2 \right) \leq 18 \left( \min\{|u(x+y) \pm 1|^2\} + \min\{|u(y) \pm 1|^2\} \right).
   \]
    It follows  from \eqref{eq: est diff u} (with the notation for $(x,y)$ as in \eqref{eq:splitinterpol})
that 
    \begin{equation}\label{eq:interpolI2}
        \begin{split}
         &  \int_{\Omega_2}\frac{|u(x+y) - u(y)|^2}{|x|^{d+2s}}\,dx\,dy
            \leq 18\int_{\Omega_2} \frac{\min\{|u(x+y)\pm 1|^2\} + \min\{|u(y) \pm 1|^2\}}{|x|^{d+2s}}\,dx\,dy\\
            &\leq 36\int_{\{ |x|\geq \delta\}} \int_{\Pi^d}\frac{\min\{|u(w)\pm1|^2\}}{|x|^{d+2s}} \,dw\,dx \leq 
            36 C(d)
            \left(\frac{1}{2s}\left(\delta^{-2s}-1\right)\right)\int_{\Pi^d} \min\{|u(w)\pm1|\}^2 \,dw\\
            &
            \leq \delta^{1-2s} \frac{C(d,s)}{c_W}\frac{1}{\delta}\int_{\Pi^d} W(u)\,d\calL^d\leq C(d,s,W)\max\{1,\delta^{1-2s}\}\int_{\Pi^d} \frac{1}{\delta}W(u)\,d\calL^d.
        \end{split}
    \end{equation}
  Finally, consider the contribution from $\Omega_3$. Note that for all points $(x,y)\in\Omega_3$ we have $1/4\leq |u(x+y) - u(y)|\leq 4$ which implies that $|u(x+y)-u(y)|^2\leq 4|u(x+y)-u(y)|$. Moreover, note that it holds for all $a,b \in \R$ with $|a-b| \geq 1/4$ that 
  \begin{align*}
  \int_a^b \min\{|t \pm 1|\} \, dt \geq \int_{7/8}^{9/8} |t - 1| \, dt = 2 \int_0^{1/8} t \, dt = \frac1{64}.
  \end{align*}
  Consequently,
  \begin{eqnarray}\label{eq:Omega3}
      |u(x+y)-u(y)|&\leq& 4\leq 256 \int_{u(y)}^{u(x+y)}\min\{ | t \pm 1 |\}\,dt \leq \frac{256}{\sqrt{c_W}}  \int_{u(y)}^{u(x+y)}\sqrt{W}\,d\calL^1\nonumber\\
      &\leq& \frac{256}{\sqrt{c_W}} |x| \int_0^{1}\left(\frac{1}{\delta}W(u(y+\tau x))+\delta|\nabla u(y+\tau x)|^2\right)\,d\tau.
  \end{eqnarray}
  We now distinguish the cases $s\neq 1/2$ and $s=1/2$. If $s\neq 1/2$ we obtain with \eqref{eq:Omega3}
\begin{equation}\label{eq:interpolI3}
        \begin{split}
           & \int_{\Omega_3}\frac{|u(x+y) - u(y)|^2}{|x|^{d+2s}}\,dy \,dx\leq 4 \int_{\Omega_3}\frac{|u(x+y)-u(y)|}{|x|^{d+2s}}\,dy\,dx\\
            &\leq \frac{1024}{\sqrt{c_W}}  \int_{\Pi^d}\int_{\{ |x| \in (\delta,\sqrt{d}/2)\}} \frac{|x|}{|x|^{d+2s}}\left(\int_{0}^1\left(\frac{1}{\delta}W(u(y+\tau x))+\delta |\nabla u|^2(y+\tau x)\right) \,d\tau\right)\,dx\,dy\\
            &\leq \frac{C(d)}{ \sqrt{c_W}}  
            \left(\int_{|x|\in (\delta,\sqrt{d}/2)}\frac{1}{|x|^{d+2s-1}}\,dx\ \right)
            \int_{\Pi^d} \left(\frac{1}{\delta}W(u)+\delta |\nabla u|^2\right)\,d\calL^d.
        \end{split}
    \end{equation}
    We use again that for $\delta\leq 1/2$, 
    \[\int_{\sqrt{d}/2\geq |x|\geq \delta}\frac{1}{|x|^{d+2s-1}}\,dx\leq \frac{C(d)}{|1-2s|}\max\{1,\delta^{1-2s}\}, \]
    which is uniformly bounded in $\delta$ if $s\in (0,1/2)$. Hence, 
    inserting \eqref{eq:interpolI1}, \eqref{eq:interpolI2}, and \eqref{eq:interpolI3} into \eqref{eq:splitinterpol} yields the first and third assertions. 
    If $s=1/2$, the assertion follows similarly using $d+2s-1=d$ and 
    \[\int_{\sqrt{d}/2 \geq|x|\geq \delta} \frac{1}{|x|^d}\,dx \leq C(d) |\ln(\delta)|.\]

\end{proof}
We note that following the lines of the proof of Proposition \ref{prop:interpolation}, one can also obtain the nonlinear interpolation inequalities in the non-periodic setting stated in Theorem \ref{th:interpol}.

\subsection{A nonlocal interpolation inequality}

It turns out that the estimates obtained in Section \ref{sec:nonlininterpol1} are sharp but not enough to prove the lower bound for the energy functional \eqref{eq:functional}. We note that for $s=1/2$, the estimates from the previous section would give us that for arbitrary $\kappa,\sigma>0$, we have for all $u\in W^{1,2}_{\text{per}}$ and $\delta\in(0,1/2)$
\begin{eqnarray*}
\sum_{k\in\Z}\min\left\{\kappa \frac{k^2}{\sigma},\ 1\right\}|\hat{u}_k|^2&\leq& \sum_{k \in \Z} \left(\frac{\kappa}{\sigma}\right)^{1/2}|k| |\hat{u}_k|^2=\left(\frac{\kappa}{\sigma}\right)^{1/2}|u|_{H^{1/2}}^2\\
&\leq & \left(\frac{\kappa}{\sigma}\right)^{1/2}|\ln\delta|\int_0^1\left(\frac{1}{\delta}W(u)+\delta |u'|^2\right)\,d\calL^1.
\end{eqnarray*}
In this section we prove Proposition \ref{prop:basicinterpolwithoutdelta} which, as explained in Section \ref{sec:mainresult}, improves the above estimate to avoid the logarithmic term in $\delta$ under mild assumptions on $W$, c.f.~Assumption \ref{ass:W}.
Some of the techniques in the proof are inspired by the works \cite{ChoksiKohnOtto99,Choksi01}, where the authors bound $\| u \|_{L^2}^2$ by $\int_{\Omega}\left( \frac{1}{\delta} W(u) + \delta |\nabla u|^2\right) \, d\calL^d$ with a small prefactor and a nonlocal term which corresponds roughly to a negative Sobolev norm.

\begin{proof}[Proof of Proposition \ref{prop:basicinterpolwithoutdelta}]
Let $L\in\N$ and $M>0$. By density it suffices to consider $u \in W^{1,2}_{\text{per}}(\Pi^d) \cap C^1(\Pi^d)$. We will use a continuous, piecewise affine approximation of $u$. 
For that, we decompose the cube $\Pi^d$ into $L^d$ cubes of side length $L^{-1}$, and consider an associated regular triangulation of the cube $\Pi^d$ into $d! L^d$ regular simplices (see Figure \ref{fig:triangulation}). 
\begin{figure}[h]
\includegraphics[width=3cm]{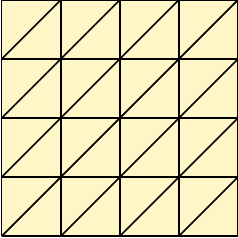}
\caption{Example of a regular triangulation of $\Pi^2$.}
\label{fig:triangulation}
\end{figure}
Note that for each simplex $T$ of the triangulation it holds that 
\begin{equation*}
	\begin{split}
		\diam(T) = \frac{\sqrt{d}}{L} , \quad \calL^d(T) = \frac{1}{d!L^d}.  
	\end{split}
\end{equation*}
By definition, each $T_i$ contains $d+1$ vertices, more precisely, $T_i$ is the convex hull of $(d+1)$ vertices. Note that there are in total $(L+1)^d$ vertices in the triangulation, and we denote the (finite) set of vertices $\calV$, and the set of simplices by $\calT$. We define a periodic piecewise affine approximation $\u^{(L)}$ of $u$ based on this triangulation, inspired by a Clement-type approximation (see e.g. \cite{carstensen06} and the references therein). For this, it suffices to define the function values at each vertex. For each vertex $v\in [0,1)^d$ of the triangulation, let  $T_{v}$ be an arbitrary simplex of the triangulation that contains $v$, and let $u^{(L)}(v)$ be the average of $u$ on the simplex $T_{v}$, i.e., 
\begin{equation*}
	u^{(L)}(v) := \intbar_{T_{v}} u\,d\calL^d.
\end{equation*} 
If one of the coordinates of $v$ is equal to one, the value $u^{(L)}(v)$ is determined by $\Z^d$-periodicity of $u^{(L)}$. 
 Using barycentric coordinates on $T\in\calT$, we can represent every $x\in T$ as convex combination $x=\sum_{j=1}^{d+1}\theta_j^{(T)}(x)v_j^{(T)}$ with $\theta_j^{(T)}(x) \in[0,1]$, $\sum_{j=1}^{d+1}\theta_j^{(T)}(x)=1$, and $T=\text{conv}\{v_j^{(T)}:j=1,\dots, d+1\}$ with $v_j^{(T)}\in\calV$. Then
\begin{eqnarray}\label{eq:estderivbary}
u^{(L)}(x):=\sum_{j=1}^{d+1}\theta_j^{(T)}(x)\,u_L(v_j^{(T)}).
\end{eqnarray}
We denote the Fourier coefficients of $u^{(L)}$ by $\hat{u}^{(L)}_k$, $k\in\Z^d$, and observe
\begin{eqnarray}\label{eq:pwaffdecomp}
&&\sum_{k\in\Z^d}\min\{1,\frac{|k|^2}{M^2}\}|\hat{u}_k|^2\leq2\sum_{k\in\Z^d}\min\{1,\frac{|k|^2}{M^2}\}\left(|\hat{u}_k-\hat{u}^{(L)}_k|^2+|\hat{u}_k|^2\right)\nonumber\\
&\leq& 2\left(\sum_{k\in\Z^d}|\hat{u}_k-\hat{u}^{(L)}_k|^2+\sum_{k\in\Z^d}\frac{|k|^2}{M^2}|\hat{u}_k^{(L)}|^2\right)=2\left(\left\|u-u^{(L)}\right\|_{L^2(\Pi^d)}^2+\frac{1}{M^2}\left\|\nabla u^{(L)}\right\|_{L^2(\Pi^d)}^2\right).
\end{eqnarray}
We estimate the two terms on the right-hand side of \eqref{eq:pwaffdecomp} separately. First, using \eqref{eq:estderivbary} and Jensen's inequality, we obtain
\begin{eqnarray}\label{eq:nonlocinter1}
\left\|u-u^{(L)}\right\|_{L^2(\Pi^d)}^2&=& \sum_{T\in\calT}\int_{T}\left|u-u^{(L)}\right|^2\,d\calL^d\leq \sum_{T\in\calT}\int_{T}\left|\sum_{j=1}^{d+1}\theta_j^{(T)}(x)\left(u(x)-u^{(L)}(v_j^{(T)})\right)\right|^2\,dx\nonumber\\
&=& \sum_{T\in\calT}\int_{T}\left|\sum_{j=1}^{d+1}\theta_j^{(T)}(x)\intbar_{T_{v_j^{(T)}}}\left(u(x)-u(y)\right)\,dy\right|^2\,dx\nonumber\\
&\leq& (d+1)\sum_{T\in\calT}\int_{T}\sum_{j=1}^{d+1}\intbar_{T_{v_j^{(T)}}}\left(u(x)-u(y)\right)^2\,dy\,dx.
\end{eqnarray}
We now use Assumption \ref{ass:W}, and let $\phi$ be as in (H3). We set $\psi:=\phi\circ u$. For vectors $a,b\in\R^d$, we denote by $[a,b]$ the line connecting $a$ and $b$. Then
\begin{eqnarray}\label{eq:estnonlocdiff}
&&(d+1)\sum_{T\in\calT} \sum_{j=1}^{d+1}\int_{T}\intbar_{T_{v_j^{(T)}}}\left(u(x)-u(y)\right)^2\,dy\,dx \leq (d+1)c_W'\sum_{v\in\calV}\int_{T_v}\intbar_{T_{v_j}}\left|\psi(x)-\psi(y)\right|\,dy\,dx\nonumber\\
&\leq&(d+1)c_W'\sum_{T\in\calT} \sum_{j=1}^{d+1}\int_{T}\intbar_{T_{v_j^{(T)}}}\left(\left|\psi(x)-\psi(v_j^{(T)})\right|+\left|\psi(v_j^{(T)})-\psi(y)\right|\right)\,dy\,dx\nonumber\\
&\leq& (d+1)c_W'\sum_{T\in\calT} \sum_{j=1}^{d+1}\int_{T}\intbar_{T_{v_j^{(T)}}}\left(\int_{[x,v_j^{(T)}]}|\nabla \psi(z)|\,dz+\int_{[v_j^{(T)},y]}|\nabla \psi(z)|\,dz\right)\,dy\,dx\nonumber\\
&\leq& (d+1)c_W'\sum_{T\in\calT} \sum_{j=1}^{d+1}\diam(T)\left(\int_{T}|\nabla \psi(z)|\,dz+\int_{T_{v_j^{(T)}}}|\nabla \psi(z)|\,dz\right),
\end{eqnarray}
where we used that all simplices $T\in\calT$ have the same diameter. Hence, combining \eqref{eq:nonlocinter1} and \eqref{eq:estnonlocdiff}, we obtain
\begin{eqnarray}\label{eq:nonloc1}
\left\|u-u^{(L)}\right\|_{L^2(\Pi^d)}^2&\leq& (d+1)c_W' \frac{d^{1/2}}{L}\sum_{T\in\calT}\left(\int_{T}|\nabla \psi(z)|\,dz+\int_{T_{v_j^{(T)}}}|\nabla \psi(z)|\,dz\right)\nonumber\\
&\leq& 2(d+1)^2c_W' \frac{d^{1/2}}{L}\int_{\Pi^d}|\nabla \psi(z)|\,dz.
\end{eqnarray}
We now turn to the second term of the right-hand side of \eqref{eq:pwaffdecomp}. We observe  that for all $T= \text{conv}\{v_j^{(T)}:j=1,\dots, d+1\}\in\calT$ and all $x\in \text{int}(T)$, there holds
\[|\nabla u_L(x)|\leq \sum_{j,k=1}^{d+1}\frac{|u_L(v_k^{(T)})-u_L(v_j^{(T)})|}{|v_k^{(T)}-v_j^{(T)}|}\leq L\sum_{j,k=1}^{d+1}|u_L(v_k^{(T)})-u_L(v_j^{(T)})|.\]
Furthermore, for any $T,T'\in\calT$ there is an isometry $R_{T,T'}:T\to T'$, and hence
\begin{eqnarray}\label{eq:nonloc21}
&&\left\|\nabla u^{(L)}\right\|_{L^2(\Pi^d)}^2=\sum_{T\in\calT}\int_{T} \left|\nabla u^{(L)}\right|^2\,d\calL^d\leq L^2\sum_{T\in\calT} \int_{T}\left(\sum_{j,k=1}^{d+1}|u_L(v_k^{(T)})-u_L(v_j^{(T)})|\right)^2\,d\calL^d\nonumber\\
&=&(d+1)L^2\sum_{T\in\calT}\sum_{j,k=1}^{d+1}\left(\intbar_{T_{v_j^{(T)}}}\left(u-u\circ R_{T_{v_j^{(T)}},T_{v_k^{(T)}})}\right)\,d\calL^d\right)^2\nonumber\\
&\leq& (d+1)L^2\sum_{T\in\calT}\sum_{j,k=1}^{d+1}\intbar_{T_{v_j^{(T)}}}\left(u-u\circ R_{T_{v_j^{(T)}},T_{v_k^{(T)}})}\right)^2\,d\calL^d\nonumber\\
&\leq&(d+1)c_W'L^2\sum_{T\in\calT}\sum_{j,k=1}^{d+1}\intbar_{T_{v_j^{(T)}}}\left|\psi-\psi\circ R_{T_{v_j^{(T)}},T_{v_k^{(T)}})}\right|\,d\calL^d\nonumber\\
&\leq& 2(d+1)^3c_W'L^2\sum_{T\in\calT} 4\diam(T)\int_{\Pi^d}|\nabla \psi|\,d\calL^d,
\end{eqnarray}
where the last step follows similarly to \eqref{eq:estnonlocdiff}.
Therefore, combining \eqref{eq:pwaffdecomp}, \eqref{eq:nonloc1} and \eqref{eq:nonloc21}, we conclude
\begin{eqnarray*}
\sum_{k\in\Z^d}\min\{1,|k|^2/M^2\}|\hat{u}_k|^2&\leq& c\left(\frac{1}{L}+\frac{L}{M^2}\right)\int_{\Pi^d}|\nabla \psi|\,d\calL^d=c\left(\frac{1}{L}+\frac{L}{M^2}\right)\int_{\Pi^d}\sqrt{W(u)}|\nabla u|\,d\calL^d\\&\leq& c\left(\frac{1}{L}+\frac{L}{M^2}\right)\int_{\Pi^d}\left(\frac{1}{\delta}W(u)+\delta |\nabla u|^2\right)\,d\calL^d.
\end{eqnarray*}
\end{proof}
We state the following consequence of Proposition \ref{prop:basicinterpolwithoutdelta}, which will be used in the proof of the lower bound.
\begin{corol}\label{cor:lbinterpol}
There is a constant $c_{int}>0$ such that for all $b,\sigma,\kappa,\Lambda>0$ with $\sigma\geq \kappa$ and all $(u,h)\in W^{1,2}_{\text{per}}\times W^{2,2}_{\text{per}}$ and all $\delta>0$ there holds
\begin{eqnarray*}
\int_0^1\left(\frac{\kappa}{2}|h''|^2+\frac{\sigma}{2}|h'|^2+\Lambda uh''\right)\,d\calL^1\geq -c_{int} \frac{\Lambda^2}{(\kappa\sigma)^{1/2}}\int_0^1\left(\frac{1}{\delta}W(u)+\delta|u'|^2\right)\,d\calL^1.
\end{eqnarray*}
\end{corol}
\begin{proof}
Let $(u,h)\in W^{1,2}_{\text{per}}\times W^{2,2}_{\text{per}}$, and denote by $\hat{u}_k$ and $\hat{h}_k$ the Fourier coefficients of $u$ and $h$, respectively. Then, optimizing in $\hat{h}_k$, we obtain 
\begin{eqnarray*}
\int_0^1\left(\frac{\kappa}{2}|h''|^2+\frac{\sigma}{2}|h'|^2+\Lambda uh''\right)\,d\calL^1&=&\sum_{k\in\Z}\left(\left(\frac{\kappa}{2}k^4+\frac{\sigma}{2}k^2\right)|\hat{h}_k|^2+\Lambda k^2 \hat{u}_k\hat{h}_k\right)\\&\geq&-\sum_{k\in\Z}\frac{\Lambda^2 k^2}{2(\kappa k^2+\sigma)}|\hat{u}_k|^2
\geq-\frac{\Lambda^2}{2\kappa}\sum_{k\in\Z}\min\{1,\,\frac{\kappa}{\sigma}k^2 \}|\hat{u}_k|^2.
\end{eqnarray*}
We now apply Proposition \ref{prop:basicinterpolwithoutdelta} with $M:=(\sigma/\kappa)^{1/2} \geq 1$ and $L:=\lfloor M\rfloor$ to conclude
\begin{eqnarray*}
\int_0^1\left(\frac{\kappa}{2}|h''|^2+\frac{\sigma}{2}|h'|^2+\Lambda uh''\right)\,d\calL^1\geq -c_{int}\frac{\Lambda^2}{(\kappa\sigma)^{1/2}}\int_0^1\left(\frac{1}{\delta}W(u)+\delta|u'|^2\right)\,d\calL^1.
\end{eqnarray*}
\end{proof}

\section{Proof of the scaling law}\label{sec:scaling}
In this section, we prove the scaling law stated in Theorem \ref{maintheo}. The structure is as follows:
We prove the upper and lower bounds separately in Subsections \ref{sec:ub} and \ref{sec:lb}, respectively. Precisely, the upper bound in \eqref{eq:sc1} is proven in Proposition \ref{prop:ub}, the lower bound follows directly from  Proposition \ref{prop:lbyoung}. Furthermore, we show in Proposition \ref{prop:ub1} that there is a constant $C>0$ (depending on $W$) such that for all parameters $\Lambda, \kappa, \sigma, b > 0$, we have admissible test functions $(u,h)$ with 
$ \mathcal{F}(u,h) \leq C \min\{b^{1/2},1\}$, 
which in particular implies the upper bound in \eqref{eq:sc2}. The lower bound in \eqref{eq:sc2} is proven in Proposition \ref{prop:lb2}.

\subsection{Upper bound}\label{sec:ub}
We start with the upper bounds in Theorem \ref{maintheo}. 
\begin{proposition}
\label{prop:ub1}
For all parameters  $\Lambda, \kappa,\sigma, b>0$ and all $W$ satisfying $(H1)$ there exists $(u,h)\in\calA$ such that
\begin{eqnarray*}
    \mathcal{F}(u,h)\leq \left(8+ \max_{[-1,1]}W\right) \min\{1,\,b^{1/2}\}.
\end{eqnarray*}
\end{proposition}
\begin{proof}
If $b\geq 1$, we use $u=h=0$, and obtain
\begin{equation}\label{eq:ub1}
    \inf_\calA\mathcal{F} \leq \mathcal{F}(0,0) = W(0).
\end{equation}
If $b<1$, we set  $h \equiv 0$ and $u$ as (see Figure \ref{fig: construction1})
\begin{equation*}
    u(x) := \begin{cases} 
    \frac{4}{b^{1/2}}x - 1,& \text{ if } x\in \left(0, \frac{b^{1/2}}{2}\right);\\
    1,& \text{ if } x \in \left[\frac{b^{1/2}}{2} , \frac{1}{2}\right);\\
    -\frac{4}{b^{1/2}}x + 1 + \frac{2}{b^{1/2}},& \text{ if } x\in \left[\frac{1}{2}, \frac{1}{2} + \frac{b^{1/2}}{2} \right)\\
    -1,& \text{ if } x \in \left[\frac{1}{2} + \frac{b^{1/2}}{2},1\right).
    \end{cases}
\end{equation*}
Then $(u,h)\in\calA$ and 
\begin{equation}
\begin{split}\label{eq:ub2}
    \mathcal{F}(u,h) &= \int_0^1\left( W(u)+ \frac{b}{2}|u'|^2\right)  \,d\calL^1\leq b^{1/2} \max_{[-1,1]}W + b \int_0^{b^{1/2}/2}  \left(\frac{4}{b^{1/2}}\right)^2 \,d\calL^1\\&= b^{1/2} \max_{[-1,1]}W + 8b^{1/2} = (8+\max_{[-1,1]}W) b^{1/2}.
\end{split}
\end{equation}
Combining \eqref{eq:ub1} and \eqref{eq:ub2} shows the assertion. 
\end{proof}
To obtain the upper bound $\sim-\Lambda^2/\kappa$, we use more complex structures. We proceed in two steps, and first present the construction with two parameters which subsequently are chosen in various parameter regimes.
\begin{proposition}
    \label{lem:ub}
For every $n\in\N$, every $\eps\in(0,1]$, all $W$ satisfying $(H1)$ and all parameters  $\Lambda, \kappa,\sigma, b>0$, there exists an admissible pair $(u_{n,\eps},h_n) \in \cA$ such that
\begin{eqnarray}
    \label{eq:ubansatz1}
    \mathcal{F}(u_{n,\eps},h_n)\leq \max_{[-1,1]}W\,\varepsilon + \frac{8bn^2}{\varepsilon} - \frac{24\Lambda^2 (1-\eps/2)^2 n^2}{\sigma + 48 \kappa n^2}.
\end{eqnarray}
\end{proposition}
\begin{remark}
   We note that the simultaneous explicit optimization in $n$ and $\eps$ in \eqref{eq:ubansatz1} is generally not trivial. However, the last term $\sim\frac{24\Lambda^2 n^2}{\sigma + 48 \kappa n^2}$ can exhibit essentially two behaviors. For $n^2 = 1$, one obtains $\sim \frac{\Lambda^2}{\sigma + \kappa}$, whereas for $n^2 \sim \frac{\sigma}{\kappa}$ one obtains for the this term $\sim \frac{\Lambda^2}{\kappa}$. After choosing $n^2$ one can then optimize explicitly in $\varepsilon$.
    We show later that in terms of matching lower bounds these choices lead to the optimal energy up to a multiplicative constant.
\end{remark}
\begin{proof}
Let $n\in\N$ and $\eps\in (0,1]$. We use the construction sketched in Figure \ref{fig: construction2}. Precisely, consider 
\begin{equation*}
    h(x) := \begin{cases}
        \frac{1}{2}x(x-\frac{1}{2}) & \text{ if }x\in(0, 1/2),\\
        -\frac{1}{2}(x-1)(x-\frac{1}{2}) & \text{ if }x\in[1/2, 1),
    \end{cases}
\end{equation*}
and extend it periodically to $\mathbb{R}$. Let $\mu>0$ to be chosen below, and define $h_n(x):= \mu h(nx)$ and $u_{n,\eps}: [0,1] \to \R$ as
\[
u_{n,\eps}(x) = \begin{cases} 1 &\text{ if } x \in \left(\frac{k}{2n}+ \frac{\eps}{4n}, \frac{k+1}{2n} - \frac{\eps}{4n} \right) \text{ for $k$ odd}, \\
-1 &\text{ if } x \in \left(\frac{k}{2n}+ \frac{\eps}{4n}, \frac{k+1}{2n} - \frac{\eps}{4n} \right) \text{ for $k$ even}, \\
\frac{4n}{\eps}x - \frac{2k}{\eps} &\text{ if } x \in \left(\frac{k}{2n} - \frac{\eps}{4n}, \frac{k}{2n} + \frac{\eps}{4n} \right) \text{ for $k$ odd}, \\
-\frac{4n}{\eps}x + \frac{2k}{\eps} &\text{ if } x \in \left(\frac{k}{2n} - \frac{\eps}{4n}, \frac{k}{2n} + \frac{\eps}{4n} \right) \text{ for $k$ even}.
\end{cases}
\]
Then $(h_n,u_n)$ is admissible. Note that  outside the ''transition layers'', we have $u'=0$, and in the transition layers $|u'|=\frac{4n}{\eps}$.  As there are $2n$ such  layers, each of length $\frac{\varepsilon}{2n}$, we have 
\[
\int_0^1\left( W(u_{n,\eps}) + \frac{b}2 |u_{n,\eps}'|^2 \right)\, d\calL^1 \leq \eps \max_{[-1,1]} W + \frac{b}2 \eps \left( \frac{4n}{\eps} \right)^2 = \eps \max_{[-1,1]} W + 8b \frac{n^2}{\eps}
\]
and 
\[
-\int_0^1 h_n'' u_{n,\eps} \, d\calL^1 = (1-\eps) \mu n^2 + 4n \int_0^{\frac{\eps}{4n}} \mu n^2 \frac{4n}{\eps} x \, dx = (1-\eps) \mu n^2 + \mu \frac{\eps n^2}{2}.
\]
Hence, 
we obtain
\begin{align*}
\cF(u_{n,\eps},h_n) &\leq  \eps \max_{[-1,1]} W + 8b \frac{n^2}{\eps} + \frac{\sigma \mu^2 n^2}{96} + \frac{\kappa}2 \mu^2 n^4  - \Lambda (1-\eps/2)\mu n^2.
\end{align*}
We now optimize in $\mu$, i.e., we choose $\mu := \frac{\Lambda\left(1-\eps/2\right)}{\kappa n^2 + (\sigma / 48)}$, and conclude 
\begin{equation*}
    \mathcal{F}(u, h) \leq \max_{[-1,1]}W\varepsilon + \frac{8bn^2}{\varepsilon} - \frac{24\Lambda^2 (1-\eps/2)^2 n^2}{\sigma + 48 \kappa n^2}.
\end{equation*}

\end{proof}

\begin{remark}\label{rem:youngopt}
We note that Proposition \ref{prop:ub} implies that the lower bound obtained in Proposition \ref{prop:lbyoung} is optimal. Precisely, we claim that  for all $\kappa, \sigma, \Lambda > 0$ and $W$ satisfying $(H1)$ there holds
    \begin{equation*}
        \lim_{b\rightarrow 0}\  \inf_{\calA} \mathcal{F} = -\frac{\Lambda^2}{2\kappa}.
    \end{equation*}
    Indeed, this follows from Proposition \ref{prop:ub} for $b\to 0$ by choosing 
    $\eps:=b^{1/2}$ and $n:=\lfloor b^{-1/8}\rfloor$ in Proposition \ref{prop:lb2}.
\end{remark}
We now prove the upper bound in Theorem \ref{maintheo}, part 2. 
\begin{proposition}\label{prop:ub}
    Let $W$ satisfy $(H1)$. Then there exists $C>0$ such that  for all $b,\sigma, \kappa,\Lambda > 0$ satisfying $\Lambda^2 \geq C \max\{ \sigma b, b \kappa, (b \sigma \kappa)^{1/2}, b^{1/2} \kappa  \}$ there exists an admissible pair $(u,h)$ such that 
    \[
\mathcal{F}(u,h) \leq - \frac1{20} \frac{\Lambda^2}{\kappa}.
    \]
\end{proposition}

\begin{proof}
Write $K := \max_{[-1,1]} W$. Then it follows from Proposition \ref{lem:ub} that for all $n \in \N$ and $\eps \in (0,1]$ there exist  admissible pairs $(u_{n,\eps},h_n) \in \cA$ such that 
\[
\mathcal{F}(u_{n,\eps},h_n) \leq K \eps + \frac{8bn^2}{\eps} - \frac{1}{8} \frac{\Lambda^2 n^2}{\sigma + \kappa n^2}.
\]
Now assume that $\Lambda^2 \geq \bar{C} \max\left\{ b \sigma, b \kappa, (b\sigma \kappa)^{1/2}, b^{1/2} \kappa \right\}$, where $\bar{C} = \max\{ 2048, 256\sqrt{2K} \}$.
In the following we will distinguish several cases in which we choose $\eps$ and $n$ appropriately to obtain the assertion.\\
First, we assume that $\sigma \leq \kappa$. In this case, we set $n = 1$. This means that we obtain for all $\eps \in (0,1]$ an admissible pair $(u_{1,\eps},h_1) \in \cA$ such that
\[
\mathcal{F}(u_{1,\eps},h_1) \leq K \eps + \frac{8b}{\eps} - \frac{1}{8} \frac{\Lambda^2}{\sigma + \kappa} \leq K \eps + \frac{8b}{\eps} - \frac{1}{16} \frac{\Lambda^2}{\kappa}.
\]
If $K \leq 8b$ then set $\eps = 1$ to obtain 
\[
\mathcal{F}(u_{1,1},h_1) \leq K + 8b - \frac{1}{16} \frac{\Lambda^2}{\kappa} \leq 16 b - \frac{1}{16} \frac{\Lambda^2}{\kappa} \leq - \frac{1}{32} \frac {\Lambda^2}{\kappa}
\]
since $\Lambda^2 \geq 512 b \kappa$. 
If on the other hand $8b \leq K$ then set $\eps_\ast = \frac{b^{1/2} 8^{1/2}}{K^{1/2}} $ to obtain  
\[
\mathcal{F}(u_{1,\eps_\ast},h_1) \leq 2 \sqrt{8K}  b^{1/2}  - \frac{1}{16} \frac{\Lambda^2}{\kappa} \leq - \frac{1}{32} \frac{\Lambda^2}{\kappa}
\]
since $\Lambda^2 \geq 256 \sqrt{2K} b^{1/2}  \kappa$. 
Secondly, we assume that $\kappa \leq \sigma$. In this case we set $n_\ast = \lceil \frac{\sigma^{1/2}}{\kappa^{1/2}} \rceil$. In particular, it holds $\frac{\sigma}{\kappa} \leq n_\ast^2 \leq 4 \frac{\sigma}{\kappa}$. Hence, for every $\eps \in (0,1]$ there exists an admissible pair $(u_{n_\ast,\eps},h_{n_\ast}) \in \cA$ satisfying 
\[
\mathcal{F}(u_{n_\ast,\eps},h_{n_\ast}) \leq K \eps + \frac{32 b \sigma}{\eps \kappa} - \frac{1}{16} \frac{\Lambda^2 }{\kappa}. 
\]
If $K \leq \frac{32 b \sigma}{\kappa}$ then we set $\eps = 1$ to obtain 
\[
\mathcal{F}(u_{n_\ast,1},h_{n_\ast}) \leq \frac{64 b \sigma}{\kappa} - \frac{1}{16} \frac{\Lambda^2}{\kappa} \leq - \frac1{32} \frac{\Lambda^2}{\kappa}
\]
since $\Lambda^2 \geq 2048 b \sigma$.
Eventually, if $K \geq \frac{32 b \sigma}{\kappa}$ then set $\eps^\ast =  \sqrt{\frac{32 b \sigma}{K \kappa}} \in (0,1]$ to obtain 
\[
\mathcal{F}(u_{n_\ast,\eps^\ast},h_{n_\ast}) \leq 2 \sqrt{\frac{32K b \sigma}{\kappa}} - \frac1{16} \frac{\Lambda^2}{\kappa} \leq -\frac1{32} \frac{\Lambda^2}{\kappa}
\]
since $\Lambda^2 \geq 256 \sqrt{2K} (b \sigma \kappa)^{1/2}$.
\end{proof}

\subsection{Lower Bound}\label{sec:lb}
We now prove the lower bounds in Theorem \ref{maintheo}. Note that the lower bound in \eqref{eq:sc1} follows directly from Proposition \ref{prop:lbyoung}, and it remains to prove \eqref{eq:sc2}.

\begin{proposition}\label{prop:lb2}
Let $W$ satisfy $(H1)$ and $(H2)$.
Then there exists $c>0$ such that for all $b,\sigma, \kappa,\Lambda > 0$ satisfying $\Lambda^2 \leq c \max\{b\sigma, b \kappa, (b\sigma\kappa)^{1/2},  b^{1/2}\kappa\}$ it holds 
        \begin{equation*}
            c \min\{1, b^{1/2} \} \leq \inf_{\calA} \mathcal{F}.
        \end{equation*}
 \end{proposition}

\begin{proof}
 Let $(u, h) \in \calA$.
First, as in the proof of Corollary \ref{cor:lbinterpol} it holds 
\begin{equation}\label{eq: h optimized1}
\int_0^1\left(\frac{\kappa}{2}|h''|^2+\frac{\sigma}{2}|h'|^2+\Lambda uh''\right)\,d\calL^1 \geq -\frac{\Lambda^2}{2\kappa}\sum_{k\in\Z}\min\{1,\,\frac{\kappa}{\sigma}k^2 \}|\hat{u}_k|^2.
\end{equation}
In particular, we obtain that 
\begin{equation} \label{eq: h optimized}
\int_0^1\left(\frac{\kappa}{2}|h''|^2+\frac{\sigma}{2}|h'|^2+\Lambda uh''\right)\,d\calL^1 \geq -\frac{\Lambda^2}{2\kappa} \| u \|_{L^2}^2.  
\end{equation}
Let us now assume that $\Lambda^2 \leq \bar{c} \max\{b\sigma, b \kappa, (b\sigma\kappa)^{1/2},  b^{1/2}\kappa\}$ for $\bar{c} = \min\left\{1/2, 1/(c_{int}\sqrt{8}), c_{MM} / \sqrt{2} \right\}$, where $c_{int}$ and $c_{MM}$ are the constants from Corollary \ref{cor:lbinterpol} and Proposition \ref{prop:modica-mortola}, respectively. We distinguish the following cases.
\begin{enumerate}
\item If $\max\{b\sigma, b \kappa, (b\sigma\kappa)^{1/2},  b^{1/2}\kappa\} = b\sigma$ then we estimate using \eqref{eq: h optimized1} and Proposition \ref{prop:modica-mortola} 
\begin{align*}
\mathcal{F}(u,h) &\geq \int_0^1 \left(W(u) + \frac{b}2 |u'|^2 - \frac{\Lambda^2}{2 \sigma} |u'|^2 \right)\, d\calL^1 \\ &\geq   \int_0^1\left( W(u) + \frac{b}4 |u'|^2\right) \, d\calL^1 \geq \frac{c_{MM}}{2} \min\{1,b^{1/2}\}.
\end{align*}
\item If $\max\{b\sigma, b \kappa, (b\sigma\kappa)^{1/2},  b^{1/2}\kappa\} = b\kappa$ then we estimate using \eqref{eq: h optimized}, Poincar\'e's inequality and Proposition \ref{prop:modica-mortola} 
\begin{align*}
\mathcal{F}(u,h) 
&\geq \int_0^1\left( W(u) + \frac{b}2 |u'|^2 - \frac{\Lambda^2}{2\kappa} |u'|^2 \right)\, d\calL^1 \\
&\geq \int_0^1 \left(W(u) + \frac{b}4 |u'|^2\right) \, d\calL^1 \geq \frac{c_{MM}}2 \min\{1,b^{1/2}\}.
\end{align*}
\item If $\max\{b\sigma, b \kappa, (b\sigma\kappa)^{1/2},  b^{1/2}\kappa\} = (b\sigma \kappa)^{1/2}$ then we estimate using Corollary \ref{cor:lbinterpol} with $\delta = \frac{2 c_{int} \Lambda^2}{ (\sigma \kappa)^{1/2}}$, where $c_{int} > 0$ is the constant from Corollary \ref{cor:lbinterpol},
\begin{align*}
\int_0^1 \left( \frac{\kappa}2 |h''|^2 + \frac{\sigma}2 |h'|^2 + \Lambda u h'' \right) \, d\mathcal{L}^1 &\geq - \int_0^1 \left(\frac12 W(u) + 2c_{int}^2 \frac{\Lambda^4}{\sigma \kappa} |u'|^2\right) \, d\mathcal{L}^1 \\
&\geq -\frac12 \int_0^1 \left( W(u) + \frac{b}2 |u'|^2\right) \, d\mathcal{L}^1.
\end{align*}
Hence, by Proposition \ref{prop:modica-mortola} 
\[
\mathcal{F}(u,h) \geq \frac12 \int_0^1 \left( W(u) + \frac{b}2\right) \, d\mathcal{L}^1 \geq \frac{c_{MM}}{2\sqrt{2}} \min\{1,b^{1/2}\}.
\]
\item If $\max\{b\sigma, b \kappa, (b\sigma\kappa)^{1/2},  b^{1/2}\kappa\} = b^{1/2} \kappa$ then we estimate using Proposition \ref{prop:modica-mortola} and Proposition \ref{prop:lbyoung}
\begin{align*}
\mathcal{F}(u,h) \geq \frac{c_{MM}}{\sqrt{2}} \min\{1,b^{1/2}\} - \frac{\Lambda^2}{2\kappa} \geq \frac{c_{MM}}{2\sqrt{2}} \min\{1,b^{1/2}\}.
\end{align*}
For the last estimate note that in this case it holds $b^{1/2} \kappa \geq b \kappa$ which implies that $\min\{1,b^{1/2} \} = b^{1/2} \geq \bar{c} \frac{\Lambda^2}{\kappa} \geq \frac{c_{MM}}{\sqrt{2}} \frac{\Lambda^2}{\kappa}$.
\end{enumerate}
This concludes the proof of Proposition \ref{prop:lb2}.
\end{proof}

\section{Existence of Minimizers}
\label{sec:existence}
In this section, we discuss the existence of minimizers for $\cF$ depending on the parameters $b,\sigma,\kappa$ and $\Lambda$. 

\begin{proposition}\label{prop:existence}
\begin{enumerate}
    \item Let $b,\kappa > 0$ and $\sigma,  \Lambda \geq 0$. Then 
    \[
    \inf_{\calA} \cF = \min_{\calA} \cF.
    \] 
    \item Let $\kappa = 0$, and $b,\Lambda, \sigma \geq 0$. 
    If $\frac{\Lambda^2}{\sigma} \geq b$ then there are no minimizers and
    \[
    \inf_{\calA} \cF = \begin{cases}
        -\infty &\text{ if } b < \frac{\Lambda^2}{\sigma}, \\
        0 &\text{ if } b = \frac{\Lambda^2}{\sigma}.
    \end{cases}
    \] 
    If $b>\frac{\Lambda^2}{\sigma}$, then there exists a minimizer if and only if there is a minimizer $u$ of
    \[\inf_{u\in W^{1,2}_{\text{per}}}\ \int_0^1\left( W(u)+\left(b-\frac{\Lambda^2}{\sigma}\right)|u'|^2\right)\,d\calL^1 \]
    that satisfies the regularity property $u\in W^{2,2}_{\text{per}}(0,1)$. For any $C^1$-double well potential $W$ satisfying (H1) there is a constant $K>0$ such that this holds if $b>K\frac{\Lambda^2}{\sigma}$.

    \item Let $b=0$ and $\kappa,\sigma,\Lambda \geq 0$. Then 
    \[
    \inf_{\calA} \cF =  -\frac{\Lambda^2}{2 \kappa}
    \] 
    and no minimizers exist if $\Lambda >0$. Here $-\frac{\Lambda^2}{2\kappa}$ has to be understood as $-\infty$ for $\Lambda > 0$, $\kappa =0$ and as $0$ for $\Lambda = 0$, $\kappa =0$. 
\end{enumerate}
\end{proposition}
\begin{proof}
 First, note that by the compact embedding $W^{1,2}((0,1)) \hookrightarrow L^2(0,1)$ it follows for $u_k \rightharpoonup u$ in $W^{1,2}_{per,vol}$ and $h_k \rightharpoonup h$ in $W^{2,2}_{per}$ that
    \[
    \liminf_{k\to \infty} \cF(u_k,h_k) \geq \cF(u,h),
    \]
    i.e.,~$\cF$ is lower semicontinuous with respect to weak convergence. 
    Hence, 1.~follows from the direct method of the Calculus of Variations if $\cF$ is coercive on $\calA$.
    Set $f(s,\xi) \coloneqq\frac{\kappa}2 |\xi|^2 + \Lambda s \xi$. 
    Then for $|s|\leq 1$ there holds 
    \[
    |f(s,\xi)| \geq \frac{\kappa}2 |\xi|^2 - \Lambda |\xi| \geq \frac{\kappa}{4} |\xi|^2 - \frac{\Lambda^2}{\kappa},
    \]
   and hence,
    \[
\cF(u,h) \geq \int_0^1 \left(\frac{b}2 |u'|^2 + \frac{\kappa}4 |h''|^2\right) \, d\calL^1 - \frac{\Lambda^2}{\kappa}. 
    \]
    Since by periodicity of $h$ we have $\int_0^1 h' \, dx =0$, Poincar\'e's inequality yields $\| h' \|_{L^2} \leq \| h'' \|_{L^2}$. Since additionally $\int_0^1 h(x) \, dx =0$, we obtain  $\| h \|_{L^2} \leq \| h''\|_{L^2}$. 
    This means that for all $C >0$ and $(u,h) \in \calA$ such that $\cF(u,h) \leq C$ it holds 
    \[
    \| u \|_{W^{1,2}} \leq 1 + \frac{2C}b + \frac{\Lambda^2}{\kappa b} \quad \text{ and } \quad \| h \|_{W^{2,2}} \leq \frac{6}{\kappa} \left( C + \frac{\Lambda^2}{\kappa} \right). 
    \]
    Hence, for all $C>0$ the set $\{(u,h) \in \calA: \cF(u,h) \leq C \}$ is weakly precompact, i.e., $\cF$ is coercive on $\calA$. This concludes the proof of 1.
\\\noindent
     In order to show 2., let $\kappa = 0$ and $\Lambda,\sigma,b \geq 0$. \
     Note that using integration by parts and Young's inequality we find that 
     \begin{align} \label{eq: lb kappa 0}
         \cF(u,h) &= \int_0^1\left( W(u) + \frac{b}2 |u'|^2 + \frac{\sigma}2 |h'|^2 + \Lambda u h''\right) \, d\calL^1 \nonumber\\
         &= \int_0^1 \left(W(u) + \frac{b}2 |u'|^2 + \frac{\sigma}2 |h'|^2 - \Lambda u' h'\right) \, d\calL^1 \nonumber\\
         &\geq \int_0^1 \left(W(u) + \frac12 \left( b - \frac{\Lambda^2}{\sigma} \right) |u'|^2\right)  \, d\calL^1.
     \end{align}
    In the last inequality equality holds if and only if $h'(x) = \frac{\Lambda}{\sigma} u'(x)$ for almost every $x \in (0,1)$.  
    Now, consider $u(x) = \chi_{(0,1/4) \cup (3/4,1)} - \chi_{(1/4,3/4)}$, where $\chi$ denotes the indicator function taking only values $0$ and $1$,  and $u_{\eps} = u * \rho_{\eps}$ for a symmetric standard mollifier $\rho_{\eps}$.
    Then for $0< \eps < 1/8$ it holds $u_{\eps} \in W^{1,2}_{per,vol}$ with $|u_{\eps}| \leq 1$.
    If $b \leq \frac{\Lambda^2}{\sigma}$ we obtain for $h_{\eps} = \frac{\Lambda}{\sigma} u_{\eps}$ by \eqref{eq: lb kappa 0}
    \begin{equation} \label{eq: limit kappa 0}
    \cF(u_{\eps},h_{\eps}) = \int_0^1 \left(W(u_{\eps}) + \frac12 \left( b - \frac{\Lambda^2}{\sigma} \right) |u'_{\eps}|^2\right) \, d\calL^1 \stackrel{\eps \to 0}{\longrightarrow} \begin{cases} - \infty &\text{ if } b < \frac{\Lambda^2}{\sigma}, \\ 0 &\text{ if } b = \frac{\Lambda^2}{\sigma}.   
    \end{cases}
    \end{equation}
Hence, $\inf \cF = -\infty$ for $b < \frac{\Lambda^2}{\sigma}$ and clearly there are no minimizers. Now, consider $b = \frac{\Lambda^2}{\sigma}$. Note that for $u \in W^{1,2}_{per,vol}$ by the embedding $W^{1,2} \hookrightarrow C^0$ and $\int_0^1 u \, dx = 0$ there exists $x \in (0,1)$ such that $u(x) = 0$. Then the continuity of $u$ implies that $\int_0^1 W(u) \, dx > 0$.  Together with \eqref{eq: lb kappa 0} and \eqref{eq: limit kappa 0} this yields that $\inf \cF = 0$ but there are no minimizers.
Eventually, consider the case $b > \frac{\Lambda^2}{\sigma}$. Then 
\eqref{eq: lb kappa 0} shows that 
\[
\inf_{\calA} \cF = \inf_{u \in W^{1,2}_{\text{per,vol}}} \int_0^1\left( W(u) + \frac12 \left( b - \frac{\Lambda^2}{\sigma} \right) |u'|^2\right) \, d\calL^1.
\]
By the direct method of the Calculus of Variations it is straight forward to show existence of minimizers $u^* \in W^{1,2}_{per,vol}$ for the right hand side. Recall that for any minimizer $u^* \in W^{1,2}_{per,vol}$  of the right hand side above there exists $h \in W^{2,2}_{per}$ such that 
\[
\cF(u^*,h) = \int_0^1\left( W(u) + \frac12 \left( b - \frac{\Lambda^2}{\sigma} \right) |u'|^2 \right)\, d\calL^1
\]
if and only if $h' = \frac{\Lambda}{\sigma} u'$. 
Consequently, a minimizer of $\cF$ exists if and only if there exists a minimizer $u^* \in W^{2,2}_{per}$ of $\int_0^1 \left(W(u) + \frac12 \left( b - \frac{\Lambda^2}{\sigma} \right) |u'|^2 \right)\, d\calL^1$.
By the volume constraint there exists for every $u \in W^{1,2}_{per,vol}$ a point $x \in (0,1)$ such that $u(x)=0$. Now, assume that $\frac12 \left( b - \frac{\Lambda^2}{\sigma} \right) > W(0)$. 
Then it follows for a minimizer $u$ of $\int_0^1 \left(W(u) + \frac12 \left( b - \frac{\Lambda^2}{\sigma} \right) |u'|^2 \right)\, d\calL^1$  that $\int_0^1 |u'|^2 \, d\calL^1 < 1$ which implies for all $y \in (0,1)$
\[
|u(y)| = |u(y) - u(x)| \leq \left( \int_0^1 |u'|^2 \, d\calL^1 \right)^{1/2} < 1.
\]
Hence, $u$ satisfies the Euler Lagrange equation
\[
\left(b - \frac{\Lambda^2}{\sigma} \right) u'' = W'(u) + \lambda, 
\]
where $\lambda \in \R$ is a Lagrange multiplier for the volume constraint $\int_0^1 u \, d\calL^1 = 0$.
In particular, $u \in W^{2,2}$. 
By the argument before this implies that a minimizer for $\cF$ exists.
This shows 2.

     Eventually we prove 3. 
     The inequality $\inf \mathcal{F} \geq -\frac{\Lambda^2}{2\kappa}$ follows from Proposition \ref{prop:lbyoung} and the inequality $\inf \mathcal{F} \leq -\frac{\Lambda^2}{2\kappa}$ from Remark \ref{rem:youngopt} since the energy is monotone in $b$. 
To show non-existence of minimizers, let us assume that there exists a minimizing pair $(u,h) \in \calA$. Then 
\begin{eqnarray*}
 -\frac{\Lambda^2}{2\kappa}=   \mathcal{F}^b (u,h) \geq \int_0^1\left(\frac{\kappa}{2}|h''|^2+\Lambda uh''\right)\,d\calL^1
 \geq \int_0^1\left(\frac{\kappa}{2}|h''|^2-\Lambda |h''|\right)\,d\calL^1\geq -\frac{\Lambda^2}{2\kappa},
\end{eqnarray*}
where the last inequality follows from the fact that $\min_{y\in\R}\left(\frac{\kappa}{2}y^2-\Lambda y\right)=-\frac{\Lambda^2}{2\kappa}.$
In particular, we have 
\begin{eqnarray*}
    W(u(t)) + \frac{\sigma}{2}|h'(t)|^2 =0\text{\qquad and\qquad} \frac{\kappa}{2}|h''(t)|^2 + \Lambda u(t)h''(t) = -\frac{\Lambda^2}{2\kappa} ~~~~~~~~\text{ for a.e. } t\in (0,1),
\end{eqnarray*}
which implies
\begin{eqnarray*}
        |h'(t)| = 0 \text{\qquad and\qquad}
        |h''(t)| = \frac{\Lambda}{\kappa} \text{\qquad  for a.e. } t\in(0,1),
\end{eqnarray*}
which yields a contradiction.
    \end{proof}
    
    \begin{remark}
    Let us note that some difficulties in the proof of item 2. of Proposition \ref{prop:existence} arise from our assumption that $u$ takes values only in $[-1,1]$, i.e., between the wells of $W$. If we allowed for a larger $L^\infty$-bound on $u$, we could (for small values of $b-\Lambda^2/\sigma$) use directly that minimizers of the (unconstraint) Modica-Mortola-type functional on $W^{1,2}$ satisfy the $L^\infty$-constraint (see \cite[Theorem 4.10]{murray-leoni:2016}), which would then imply that minimizers of the constraint problem are smooth enough.
    \end{remark}

    \section*{Acknowledgement}
Funded by the {\em Deutsche Forschungsgemeinschaft}  (DFG, German Research Foundation) through {\em Research Training Group} 2433 DAEDALUS (project ID 384950143), and funded by the Deutsche Forschungsgemeinschaft (DFG, German Research Foundation) under Germany´s Excellence Strategy – The Berlin Mathematics Research Center MATH+ (EXC-2046/1, project ID 390685689).

\bibliographystyle{plain}
\bibliography{andelman2}

\end{document}